\documentclass[11pt]{article}
\usepackage[T1]{fontenc}
\usepackage[utf8]{inputenc}
\usepackage{lmodern,microtype}
\usepackage[margin=1.08in]{geometry}
\usepackage{amsmath,amssymb,amsthm,mathtools}
\usepackage{enumitem}
\usepackage{fancyhdr}
\usepackage[hidelinks]{hyperref}
\usepackage[nameinlink,noabbrev]{cleveref}

\usepackage{zref-savepos}
\makeatletter
\let\NS@original@anchorstart\hyper@anchorstart
\def\NS@anchor@prefix#1.#2\@nil{#1}
\def\NS@left@equation@anchor#1{%
  \zsaveposx{NS-equation-anchor-#1}%
  \begingroup
    \edef\NS@anchor@shift{%
      \the\dimexpr 1in+\hoffset+\oddsidemargin
        -\zposx{NS-equation-anchor-#1}sp\relax}%
    \kern\NS@anchor@shift
    \NS@original@anchorstart{#1}%
    \kern-\NS@anchor@shift
  \endgroup
  \Hy@activeanchortrue
}
\def\hyper@anchorstart#1{%
  \begingroup
    \edef\NS@anchor@name{#1}%
    \edef\NS@anchor@type{%
      \expandafter\NS@anchor@prefix\NS@anchor@name.\@nil}%
    \def\NS@equation@type{equation}%
    \ifx\NS@anchor@type\NS@equation@type
      \endgroup\NS@left@equation@anchor{#1}%
    \else
      \endgroup\NS@original@anchorstart{#1}%
    \fi
}
\makeatother

\numberwithin{equation}{section}
\newtheorem{theorem}{Theorem}[section]
\newtheorem{proposition}[theorem]{Proposition}
\newtheorem{lemma}[theorem]{Lemma}
\newtheorem{corollary}[theorem]{Corollary}
\theoremstyle{definition}

\theoremstyle{remark}
\newtheorem{remark}[theorem]{Remark}
\newcommand{\R}{\mathbb R}
\newcommand{\dd}{\,dx}
\DeclareMathOperator{\curl}{curl}
\DeclareMathOperator{\diver}{div}
\setlist{itemsep=0.2em,topsep=0.4em}
\hypersetup{
 pdftitle={Beyond the L(9/5) vorticity criterion in the stationary Navier-Stokes Liouville problem},
 pdfauthor={Wangzhe Wu},
 pdfkeywords={stationary Navier-Stokes equations, Liouville theorem, Lorentz spaces, vorticity, Bernoulli function}}

\begin{document}
\title{Beyond the $L^{9/5}$ Vorticity Criterion in the Stationary Navier--Stokes Liouville Problem}
\author{Wangzhe Wu\\
 \small School of Mathematics and Statistics, Ningbo University\\
 \small Ningbo 315211, Zhejiang, China\\
 \small \texttt{wuwz18@mail.ustc.edu.cn}}
\date{}
\maketitle

\begin{abstract}
Let $(v,p)$ be a smooth stationary Navier--Stokes solution in
$\mathbb R^3$ with $v(x)\to0$ as $|x|\to\infty$.
We prove that $v\equiv0$ if its vorticity $\omega:=\curl v$ belongs to
$L^{s,\infty}(\mathbb R^3)$ for some $9/5\le s\le1.87$.
This gives a global weak-Lorentz vorticity criterion beyond $9/5$
requiring neither smallness nor an additional Fubini-type hypothesis.
The finite Dirichlet integral is recovered from the vorticity
hypothesis rather than assumed. In particular,
$|\omega(x)|=O(|x|^{-\alpha})$ implies triviality for
$\alpha\ge300/187$, crossing the $|x|^{-5/3}$ vorticity-decay scale.
The proof combines this automatic finite-energy upgrade with a new
variable-denominator Bernoulli--vorticity quotient identity.
Its completed-square form gives nonnegative defect terms, and
localizing the concavity of the denominator creates a positive
measure on a Bernoulli level surface. A pressure-weighted convex
extension of this identity yields a quantitative normalized-vorticity
trace and quotient-gradient control. These estimates are then coupled
to the Bernoulli-gradient identity through a nonlinear feedback estimate.
\end{abstract}



\section{Introduction}
\label{sec:introduction}

 We consider the stationary incompressible Navier--Stokes equations
in $\R^3$ without external forcing,
\begin{equation}
 -\Delta v+(v\cdot\nabla)v+\nabla p=0,\qquad
 \diver v=0\quad\hbox{in }\R^3,
 \label{eq:NS}
\end{equation}
together with
\begin{equation}
 \lim_{|x|\to\infty}v(x)=0.
 \label{eq:far-field}
\end{equation}
Here $v:\R^3\to\R^3$ and $p:\R^3\to\R$ are smooth.
Define the vorticity by
\begin{align}
 \omega&:=\curl v.
 \label{eq:vorticity-definition}
\end{align}
The question whether every such solution with finite Dirichlet integral
is zero is a classical Liouville problem; see
Galdi~\cite[Chapter~X]{Galdi2011}.
In this smooth, divergence-free, zero-far-field class, the conditions
$\nabla v\in L^2(\R^3)$ and $\omega\in L^2(\R^3)$ are equivalent,
and
\begin{equation}
 \|\nabla v\|_{L^2(\R^3)}^2
 =\int_{\R^3}|\nabla v|^2
 =\int_{\R^3}|\omega|^2
 =\|\omega\|_{L^2(\R^3)}^2.
 \label{eq:intro-div-curl}
\end{equation}
Consequently, the finite-Dirichlet-integral hypothesis in the
classical conjecture can equivalently be replaced by
$\omega\in L^2(\R^3)$; see Remark~\ref{rem:div-curl-equivalence}.

The obstruction at infinity is already visible in the classical formulation of the problem.  Galdi~\cite[Remark~X.9.4 and Theorem~X.9.5, p.~729]{Galdi2011} posed the Liouville problem under the finite-Dirichlet-integral assumption and proved triviality under the additional condition
\[
v\in L^{9/2}(\mathbb R^3),
\]
using a Caccioppoli-type estimate.  The exponent $9/2$ has remained a natural borderline for criteria based solely on global Lebesgue integrability of the velocity.  While a number of results are known under supplementary Besov, negative-Sobolev, variable-exponent, or other structural assumptions, to our knowledge no Liouville theorem is presently known under the sole fixed-exponent condition
\[
v\in L^p(\mathbb R^3),
\qquad p>\frac92.
\]
Finite Dirichlet energy, on the other hand, yields only the Sobolev control
\[
v\in L^6(\mathbb R^3),
\]
which by itself does not provide the decay at infinity needed to close the global energy argument.  Thus there remains a genuine gap between the classical velocity criterion $L^{9/2}$ and the $L^6$ control supplied by the energy class.

The vorticity formulation allows us to study both energy and decay
through a first-order quantity. Motivated by the energy equivalence
\eqref{eq:intro-div-curl}, we seek conditions on the global distribution
of $\omega$ that recover finite Dirichlet energy and also control the
small-amplitude part of the vorticity.
As shown in Lemma~\ref{lem:entry}, the Biot--Savart representation first
transfers the vorticity assumption to the velocity, and the stationary
vorticity equation then improves the integrability of $\omega$.
These two estimates recover finite Dirichlet energy from the
weak-Lorentz vorticity hypothesis. The remaining obstruction at
infinity is treated through the distribution of vorticity on
Bernoulli level sets.

For a measurable scalar, vector, or tensor field $u$, we use the
weak-Lorentz quasi-norm
\begin{align}
 \|u\|_{L^{a,\infty}}
 &:=\sup_{\lambda>0}\lambda
       |\{x\in\R^3:|u(x)|>\lambda\}|^{1/a},
 \qquad 1<a<\infty.
 \label{eq:weak-norm}
\end{align}
All tensor norms are Euclidean.

\begin{theorem}\label{thm:main}
Let $(v,p)$ be a smooth solution of \eqref{eq:NS} satisfying
\eqref{eq:far-field}.  If
\begin{equation}
 \omega\in L^{s,\infty}(\R^3)
 \quad\hbox{for some }\quad \frac95\le s\le1.87,
 \label{eq:main-hypothesis}
\end{equation}
then $v=0$ and $p$ is constant.
\end{theorem}

The exponent $9/5$ arises naturally from the classical velocity criterion of Galdi.  Indeed, the Biot--Savart representation together with the Hardy--Littlewood--Sobolev inequality gives
\[
\omega\in L^{9/5}(\mathbb R^3)
\quad\Longrightarrow\quad
v\in L^{9/2}(\mathbb R^3);
\]
see Stein~\cite[Chapter~V, Section~1]{Stein1970} for the potential estimate and Chae--Wolf~\cite[Corollary~2.1]{ChaeWolf2016} for the Biot--Savart representation in the stationary Navier--Stokes setting.  Galdi's $L^{9/2}$ criterion then yields $v\equiv0$.  Thus $9/5$ is the natural vorticity exponent associated with the classical $L^{9/2}$ velocity criterion.
The weak-Lorentz norm in the theorem may be arbitrarily large.  In pointwise form,
Corollary~\ref{cor:pointwise} allows every
$\alpha\ge300/187$, and therefore crosses the
$|x|^{-5/3}$ vorticity-decay scale appearing in the classical criteria.

A second conceptual advance is that the theorem does not assume
membership in the classical energy class.  For the solutions under
consideration, the weak-Lorentz vorticity hypothesis automatically
gives
\[
 \omega\in L^{s,\infty}(\mathbb R^3)
 \quad\Longrightarrow\quad
 \omega\in L^2(\mathbb R^3)
 \quad\Longrightarrow\quad
 \int_{\mathbb R^3}|\nabla v|^2
 =\int_{\mathbb R^3}|\omega|^2<\infty;
\]
see Lemma~\ref{lem:entry} and \eqref{eq:entry-dirichlet}.  Thus the
finite Dirichlet integral is obtained as a consequence rather than
imposed as a hypothesis.  The proof accordingly separates into two stages: an automatic entry into the classical energy class, then followed by a Bernoulli-level argument that overcomes the remaining obstruction at infinity.

A power-law bound $|\omega(x)|=O(|x|^{-\alpha})$ places
$\omega$ in $L^{3/\alpha,\infty}$; thus the theorem has the following
pointwise consequence.

\begin{corollary}\label{cor:pointwise}
Under the smoothness and far-field assumptions of
Theorem~\ref{thm:main}, suppose that, for some $\alpha\ge300/187$,
\begin{equation}
 |\omega(x)|\le C(1+|x|)^{-\alpha}\qquad(x\in\R^3).
 \label{eq:pointwise}
\end{equation}
Then $v=0$.
The same conclusion holds if \eqref{eq:pointwise} is imposed only
outside a ball.
\end{corollary}

\begin{proof}
Smoothness on bounded sets and \eqref{eq:pointwise} imply
$|\omega(x)|\le C(1+|x|)^{-300/187}$ globally, even if the decay
is assumed only outside a ball. Thus
$\omega\in L^{1.87,\infty}(\R^3)$, and
Theorem~\ref{thm:main} applies.
\end{proof}

\subsection{Relation to earlier criteria}

The strong vorticity criterion $\omega\in L^{9/5}(\mathbb R^3)$
follows from Galdi~\cite[Theorem~X.9.5]{Galdi2011} through the
Biot--Savart estimate; see also
Chae--Wolf~\cite[Remark~1.2]{ChaeWolf2016}. Their
Theorem~1.3 treats more general assumptions
$\omega\in L^\sigma$ under the additional Fubini condition~(1.9).
For
\[
\frac95\le \sigma\le1.87,
\]
Theorem~\ref{thm:main} removes this supplementary condition and, at the
same time, replaces the strong Lebesgue assumption by the weaker
condition $\omega\in L^{\sigma,\infty}$.

A different pointwise criterion was obtained by Kozono, Terasawa, and
Wakasugi~\cite{KTW2017}.  Writing
\begin{align}
 E_\omega&:=\limsup_{|x|\to\infty}|x|^{5/3}|\omega(x)|,
 \label{eq:intro-far-field-vorticity}
\end{align}
their criterion gives triviality when $E_\omega$ is sufficiently small relative
to the Dirichlet integral.  Our result differs in two respects.
First, when $E_\omega<\infty$, no relative smallness assumption is
required.  Second, Corollary~\ref{cor:pointwise} allows
\[
|\omega(x)|=O(|x|^{-\alpha})
\qquad
\text{for every }\alpha\ge\frac{300}{187}.
\]
Since $300/187<5/3$, this includes the range
\[
\frac{300}{187}\le\alpha<\frac53,
\]
where the vorticity may decay strictly more slowly than
$|x|^{-5/3}$ and $E_\omega$ may already be infinite.  Thus the
pointwise consequence of Theorem~\ref{thm:main} extends beyond the
$|x|^{-5/3}$ decay scale, rather than merely removing the smallness
condition at that scale.

Jarr\'in~\cite[Theorem~1(1)]{Jarrin2020}, together with Galdi's theorem and the Lorentz embeddings, covers the velocity conditions
\[
v\in L^{9/2,\ell}(\mathbb R^3),
\qquad \ell<\infty.
\]
By the Lorentz-space Biot--Savart estimate, the corresponding vorticity conditions
\[
\omega\in L^{9/5,\ell}(\mathbb R^3),
\qquad \ell<\infty,
\]
are therefore also covered.  Theorem~\ref{thm:main} reaches the remaining weak endpoint $\omega\in L^{9/5,\infty}$ and further extends the criterion to weak-$L^s$ vorticity with $s>9/5$.  Related annular Lorentz criteria, including weak-$L^{9/2}$ conditions with additional asymptotic control, were obtained by Seregin--Wang~\cite{SereginWang2020}.

Several other Liouville criteria impose different assumptions on the
velocity or its derivatives.  These include
$\Delta v\in L^{6/5}$ in Chae~\cite[Theorem~1.2]{Chae2014},
the condition $v\in L^6\cap\mathrm{BMO}^{-1}$ in
Seregin~\cite{Seregin2016}, growth conditions on velocity potentials
in Chae--Wolf~\cite{ChaeWolf2019}, and Lebesgue, Besov, Morrey, or
related velocity conditions in
\cite{ChamorroJarrinLemarie2021,Tsai2021}.  More recent variants based
on variable-exponent or radial-component assumptions are given by
Vergara-Hermosilla
\cite{VergaraHermosilla2026Above,VergaraHermosilla2026Radial}.
These results are complementary to the present theorem, whose
hypothesis is imposed directly on the global distribution of
vorticity.

Lerner~\cite{Lerner2026} develops Wiener-algebra regularity for
finite-enstrophy solutions and obtains Liouville criteria involving
low-frequency components and the Bernoulli head pressure. In particular,
for such solutions, his Corollary~1.15 gives triviality under the additional condition
$\nabla(p+|v|^2/2)\in L^{6/5}(\mathbb R^3)$.

Bernoulli level sets and coarea arguments have previously played an
important role in stationary Navier--Stokes Liouville theory; see
Chae--Wolf~\cite{ChaeWolf2016}, Chae~\cite{Chae2025}, and
Chae~\cite{Chae2026}.  A main new ingredient of the present work is
the quotient-type differential identity
\eqref{eq:ratio-general-denominator}, obtained by dividing
$|\omega|^2$ by a variable concave function of the positive Bernoulli
quantity.  After completing the trace-free strain square, this
identity produces nonnegative bulk terms together with a positive
interface contribution on Bernoulli level surfaces.  Combined with
the nonlinear feedback mechanism developed below, it provides the
key new structure behind the improved Liouville range.

\subsection{Structure of the proof}

The proof consists of two stages: an automatic finite-energy
upgrade and a nonlinear feedback argument.

The first stage recovers the energy class from the vorticity hypothesis.
For $q:=3s/(3-s)$, the Biot--Savart estimate and the stationary
vorticity equation give
\[
 v\in L^{q,\infty},\qquad
 \omega\in L^{s,\infty}\cap L^{q/2,\infty}\cap L^2;
\]
see Lemma~\ref{lem:entry}, in particular
\eqref{eq:entry-exponents}--\eqref{eq:entry-dirichlet}. Hence
\[
 D:=\int_{\mathbb R^3}|\nabla v|^2
   =\int_{\mathbb R^3}|\omega|^2<\infty.
\]
After normalizing the pressure at infinity, set
$h:=-p-|v|^2/2$, as in \eqref{eq:entry-bernoulli-definition}.
Lemma~\ref{lem:entry-bernoulli} shows that either $v\equiv0$ or
$h>0$ throughout $\mathbb R^3$.

Assume now that $D>0$ and $9/5<s<2$. Write $W:=\omega/h$ and
$f:=(-p)/h$, and let $\mu_t$ be the measure defined in
\eqref{eq:ratio-band-definitions}. For a fixed $d>0$ to be chosen,
consider the low-vorticity cores
\[
 \Omega_t:=\{t<h<2t,\ |W|\le t^d\}.
\]
For the cores selected below, the Bernoulli-gradient identity
$\nabla h=\curl\omega-v\times\omega$, together with the
quotient-gradient bound in Lemma~\ref{lem:ratio-budget}, yields
the central estimate
\[
 \mu_t(\Omega_t)^2
 \le 8
 \left(\int_{\Omega_t}f\,d\mu_t\right)
 \left(\int_{\Omega_t}|\omega|^2\,dx\right);
\]
see \eqref{eq:core-lamb-absorption}. Our use of the
Bernoulli-gradient identity is inspired by
Lerner~\cite[formula~(3.1.2)]{Lerner2026}.
The feedback argument estimates the three quantities in this
inequality as follows.

\begin{enumerate}
\item \emph{The core mass on the left.}
Lemma~\ref{lem:convex-entropy} gives the weighted trace estimate
\eqref{eq:convex-entropy-trace}. Combined with the pressure-tail
bound in Lemma~\ref{lem:convex-tail} and the mass identity
\eqref{eq:convex-mass}, it allows us to choose $d>0$ so that
\[
 \mu_t(\Omega_t)\longrightarrow\frac D2.
\]
Thus restricting to the low-vorticity core does not lose the
positive limiting mass.

\item \emph{The pressure factor on the right.}
Lemma~\ref{lem:convex-tail}, specifically
\eqref{eq:convex-pressure-core}, gives
\[
 \int_{\Omega_t}f\,d\mu_t
 \le Ct^{-\delta/(1+\delta)}
       \mu_t(\Omega_t)^{\delta/(1+\delta)},
\]
where $\delta=(5s-9)/(2(3-s))$ is defined in
\eqref{pm:parameters}. This controls the possible growth of the
pressure integral in terms of the same core mass.

\item \emph{The vorticity factor on the right.}
On $\Omega_t$, the definition of the core gives
$|\omega|\le2t^{1+d}$. The weak-$L^s$ vorticity hypothesis
therefore implies
\[
 \int_{\Omega_t}|\omega|^2\,dx
 \le Ct^{(2-s)(1+d)};
\]
see \eqref{eq:convex-core-energy}. The small-vorticity restriction
thus supplies a quantitative decay rate.
\end{enumerate}

Substituting the last two estimates into
\eqref{eq:core-lamb-absorption} returns the core mass on the
right with the smaller power $\delta/(1+\delta)$:
\[
 \mu_t(\Omega_t)^2
 \le Ct^{(2-s)(1+d)-\delta/(1+\delta)}
       \mu_t(\Omega_t)^{\delta/(1+\delta)}.
\]
Since $\mu_t(\Omega_t)\to D/2>0$, rearranging gives
\[
 D^{2+\delta}
 \le Ct^{(1+\delta)(2-s)(1+d)-\delta}.
\]
To obtain a contradiction as $t\downarrow0$, we require
\[
 (1+\delta)(2-s)(1+d)-\delta>0.
\]
This condition must hold together with the upper bound on $d$ in
\eqref{eq:convex-d-upper}, which ensures that the core retains the
limiting mass. Their compatibility, computed in
\eqref{eq:strict-parameter-margin}, gives $9/5<s<s_*$, where
$s_*=1.870963541397\ldots$ is defined in \eqref{eq:thresholds}.
In particular, the argument covers $9/5<s\le1.87$.

The endpoint $s=9/5$ is recovered separately from the additional integrability supplied by Lemma~\ref{lem:entry}.  Indeed, \eqref{eq:entry-exponents}--\eqref{eq:entry-velocity} give
\[
\omega\in L^{9/5,\infty}(\mathbb R^3)
\cap L^{9/4,\infty}(\mathbb R^3),
\]
and the intersection of these two weak-Lorentz spaces implies
\[
\omega\in L^\sigma(\mathbb R^3)
\qquad
\text{for every }\frac95<\sigma<\frac94.
\]
Therefore the  case $s = 9/5$ also holds.

Throughout the paper, $C$ denotes a positive constant that may
change from one occurrence to another.

\section{Finite energy and normalization of the pressure}
\label{sec:entry}

For the technical argument, fix a smooth solution $(v,p)$ of
\eqref{eq:NS}--\eqref{eq:far-field} and an exponent
$s$ with $9/5\le s<2$.  We reintroduce the vorticity by
\begin{align}
 \omega&:=\curl v,
 \label{eq:entry-vorticity-definition}
 \\
 \|\omega\|_{L^{s,\infty}}
 &:=\sup_{\lambda>0}\lambda
       |\{x\in\R^3:|\omega(x)|>\lambda\}|^{1/s}.
 \label{eq:entry-weak-vorticity-norm}
\end{align}
In this section, we first derive the finite Dirichlet integral from the vorticity
hypothesis, rather than assuming it.  All the spaces in this section
are over $\mathbb R^3$ unless a domain is explicitly indicated.

The Biot--Savart and pressure representations themselves are classical.
For their use in stationary Navier--Stokes Liouville arguments, see
Chae--Wolf~\cite[Corollary~2.1]{ChaeWolf2016} for the Biot--Savart
representation and Chae~\cite[proof of~(2.29), preceding~(2.30)]{Chae2026}
for the pressure reconstruction.
We therefore cite the underlying singular-integral theorems below and
focus on the weak-Lorentz consequence needed here.  The Biot--Savart
estimate first recovers $v$ from $\omega$, and the vorticity equation
then yields a second potential estimate for $\omega$ at an exponent
strictly larger than $2$.  Together with the original assumption
$s<2$, these two bounds control respectively the large- and
small-amplitude parts of the distribution of $\omega$, and hence imply
$\omega\in L^2$ without any smallness or a priori energy assumption.
The higher exponent does not, by itself, improve the small-amplitude
distribution; that information continues to come from the original
weak-$L^s$ hypothesis.

Once finite energy has been obtained, we normalize the pressure and
derive the decay estimates for its derivatives.

We record the singular-integral facts used below.  For
$j=1,2,3$, let $\mathcal R_j$ denote the $j$th Riesz transform,
defined on Schwartz functions by the principal-value formula
\begin{equation}
 \mathcal R_j f(x)
 =\frac1{\pi^2}\operatorname{p.v.}
       \int_{\mathbb R^3}
       \frac{x_j-y_j}{|x-y|^4}f(y)\,dy.
 \label{eq:riesz-principal-value}
\end{equation}
Equivalently, with the Fourier convention
\begin{align}
 \widehat f(\xi)
 :=\int_{\mathbb R^3}e^{-ix\cdot\xi}f(x)\,dx,
\end{align}
one has
\begin{align}
 \widehat{\mathcal R_j f}(\xi)
 =-i\frac{\xi_j}{|\xi|}\widehat f(\xi),
 \qquad \xi\ne0.
 \label{eq:riesz-definition}
\end{align}
The convolution kernel is
\[
K_j(z)=\pi^{-2}z_j|z|^{-4},
\]
so that \(K_j(x-y)=\pi^{-2}(x_j-y_j)|x-y|^{-4}\). The kernel is odd, homogeneous of degree $-3$, and has zero spherical
mean; its cancellation is what gives meaning to the principal value.
We shall use only the standard $L^p$ and Lorentz-space boundedness
properties of these operators.  See
Stein~\cite[Chapter~III, Section~1]{Stein1970}.

The Calder\'on--Zygmund theorem extends each $\mathcal R_j$ to a
bounded operator on $L^a$ for $1<a<\infty$; see
\cite[Chapters~II--III]{Stein1970}.  Real interpolation between
two strong exponents on either side of $a$ gives boundedness on
$L^{a,\infty}$ as well; see
Bennett and Sharpley~\cite[Chapter~V, Sections~1--2]{BennettSharpley1988}.
In particular, for $1<a<\infty$ and $1\le i,j\le3$,
\begin{align}
 \|\mathcal R_j f\|_{L^a}
   +\|\mathcal R_i\mathcal R_j f\|_{L^a}
 &\le C_a\|f\|_{L^a},\notag\\
 \|\mathcal R_j f\|_{L^{a,\infty}}
   +\|\mathcal R_i\mathcal R_j f\|_{L^{a,\infty}}
 &\le C_a\|f\|_{L^{a,\infty}}.
 \label{eq:riesz-boundedness}
\end{align}
The endpoint extensions agree on their common domain and hence define
a single operator on the sum of the two strong spaces.  Real
interpolation then gives its bounded extension to the intermediate
weak Lorentz space,
\[
 (L^{p_0},L^{p_1})_{\theta,\infty}=L^{a,\infty},
 \qquad
 \frac1a=\frac{1-\theta}{p_0}+\frac{\theta}{p_1}.
\]
Thus no density of Schwartz functions in $L^{a,\infty}$ is required;
see Bergh--Löfström
\cite[Theorems~3.1.2 and~5.3.1]{BerghLofstrom1976}.

For the extensions just described, the double Riesz transforms satisfy
the corresponding Poisson identity in the sense of distributions:
if $f\in L^{a,\infty}$ with $1<a<\infty$, then
\begin{equation}
 -\Delta(\mathcal R_i\mathcal R_j f)
 =\partial_i\partial_j f
 \quad\hbox{in }\mathcal D'(\mathbb R^3).
 \label{eq:riesz-poisson}
\end{equation}
The same conclusion holds with $L^{a,\infty}$ replaced by $L^a$.
Moreover, the transforms commute with translations and with
distributional differentiation.  More precisely, if
$f\in L^{a,\infty}$ and $\partial_k f\in L^{b,\infty}$ for
$1<a,b<\infty$, then
\begin{equation}
 \partial_k(\mathcal R_i\mathcal R_j f)
 =\mathcal R_i\mathcal R_j(\partial_k f)
 \quad\hbox{in }\mathcal D'(\mathbb R^3).
 \label{eq:riesz-commutation}
\end{equation}
The analogous statement holds in the corresponding strong spaces.
These identities follow from the Fourier-multiplier representation
and duality; see Stein~\cite[Chapter~III]{Stein1970}.
Double Riesz transforms in stationary Navier--Stokes estimates on
Lorentz spaces are used, for example, by
Jarr\'in~\cite[formula~(24)]{Jarrin2020}.

We also use the Biot--Savart operator
\begin{align}
 \mathcal Bu&:=\operatorname{curl}(-\Delta)^{-1}u,
 \qquad
 \mathcal Bu(x)=\frac1{4\pi}\int_{\mathbb R^3}
       \frac{u(y)\times(x-y)}{|x-y|^3}\,dy,
 \label{eq:biot-savart-definition}
\end{align}
initially for smooth compactly supported vector fields.  Here
$(-\Delta)^{-1}$ is convolution with the Newtonian kernel
$(4\pi|x|)^{-1}$, component by component, and the integral in
\eqref{eq:biot-savart-definition} is the Biot--Savart integral.
For a general weak-space input, this integral is understood through
the Lorentz convolution estimate below rather than assumed to converge
absolutely pointwise.

The Biot--Savart kernel has size $C|x|^{-2}$.  More precisely,
\[
 |\mathcal Bu(x)|
 \leqslant
 \frac1{4\pi}\int_{\mathbb R^3}
 \frac{|u(y)|}{|x-y|^2}\,dy.
\]

Since $|x|^{-2}\in L^{3/2,\infty}(\mathbb R^3)$, O'Neil's Lorentz
convolution inequality \cite[Theorem~2.6]{ONeil1963}, applied to
$|x|^{-2}$ and $|u|$ with both Lorentz second indices equal to
$\infty$, gives
\[
 \|\mathcal Bu\|_{L^{q,\infty}}
 \leqslant C_a\|u\|_{L^{a,\infty}},
 \qquad
 \frac1q=\frac1a-\frac13,
\]
and hence
\[
 q=\frac{3a}{3-a},
 \qquad 1<a<3.
\]
For the derivative estimate, let $\varepsilon_{mij}$ denote the
Levi--Civita symbol in three dimensions:
\[
 \varepsilon_{mij}
 =
 \begin{cases}
  1, & (m,i,j)\text{ is an even permutation of }(1,2,3),\\
 -1, & (m,i,j)\text{ is an odd permutation of }(1,2,3),\\
  0, & \text{if two indices coincide}.
 \end{cases}
\]
Thus, for a vector field $F=(F_1,F_2,F_3)$,
\[
 (\operatorname{curl}F)_m
 =
 \sum_{i,j=1}^3
 \varepsilon_{mij}\partial_iF_j.
\]
Applying this to $F=(-\Delta)^{-1}u$ gives
\[
 (\mathcal Bu)_m
 =
 \sum_{i,j=1}^3
 \varepsilon_{mij}\,
 \partial_i(-\Delta)^{-1}u_j.
\]
By the Fourier multiplier definitions in
\eqref{eq:riesz-definition},
\[
 \partial_k\partial_i(-\Delta)^{-1}
 =\mathcal R_k\mathcal R_i
\]
on Schwartz functions, and hence for the Lorentz-space extensions in
the sense of distributions.
Hence
\[
 \partial_k(\mathcal Bu)_m
 =
 \sum_{i,j=1}^3
 \varepsilon_{mij}\,
 \partial_k\partial_i(-\Delta)^{-1}u_j
 =
 \sum_{i,j=1}^3
 \varepsilon_{mij}\,
 \mathcal R_k\mathcal R_i u_j.
\]
Therefore every component of $\nabla\mathcal Bu$ is a finite linear
combination of double Riesz transforms of the components of $u$.
The weak Lorentz boundedness of the double Riesz transforms thus
yields
\begin{equation}
 \|\mathcal Bu\|_{L^{3a/(3-a),\infty}}
       +\|\nabla\mathcal Bu\|_{L^{a,\infty}}
 \le C_a\|u\|_{L^{a,\infty}},
 \qquad 1<a<3.
 \label{eq:entry-potential-bounds}
\end{equation}
The corresponding strong estimates follow from the
Hardy--Littlewood--Sobolev theorem and the $L^a$ boundedness of the
Riesz transforms:
\begin{equation}
 \|\mathcal Bu\|_{L^{3a/(3-a)}}
       +\|\nabla\mathcal Bu\|_{L^a}
 \le C_a\|u\|_{L^a},
 \qquad 1<a<3.
 \label{eq:entry-strong-potential-bounds}
\end{equation}
See Stein~\cite[Chapter~V, Section~1]{Stein1970}.  In particular,
taking $a=2$ gives
\[
 u\in L^2
 \quad\Longrightarrow\quad
 \mathcal Bu\in L^6,
 \qquad
 \nabla\mathcal Bu\in L^2.
\]
We shall repeatedly use the following elementary consequence of the
weak-Lorentz distribution estimate.

\begin{lemma}[Weak average on a finite set]
\label{lem:entry-weak-average}
If $a>1$, $u\in L^{a,\infty}(\mathbb R^3)$, and
$\Omega\subset\mathbb R^3$
is measurable with finite measure, then
\begin{equation}
 \int_\Omega|u|\le \frac{a}{a-1}
       \|u\|_{L^{a,\infty}}|\Omega|^{1-1/a}.
 \label{eq:entry-weak-average}
\end{equation}
\end{lemma}

\begin{proof}
By the layer-cake formula,
\[
 \int_\Omega |u(x)|\,dx
 =
 \int_0^\infty
 \bigl|\{x\in\Omega:|u(x)|>\lambda\}\bigr|\,d\lambda
 \leqslant \int_0^\infty
 \min\Bigl\{
 |\Omega|,
 \bigl|\{x\in\mathbb R^3:|u(x)|>\lambda\}\bigr|
 \Bigr\} \,d\lambda.
\]

For every $\lambda>0$, the weak-Lorentz quasi-norm in
\eqref{eq:weak-norm} gives
\[
 \lambda
 \bigl|\{x\in\mathbb R^3:|u(x)|>\lambda\}\bigr|^{1/a}
 \leqslant
 \|u\|_{L^{a,\infty}}.
\]
Therefore
\begin{align}\label{10011043}
 \int_\Omega|u|
 \leqslant
 \int_0^\infty
 \min\bigl\{
 |\Omega|,
 \|u\|_{L^{a,\infty}}^a\lambda^{-a}
 \bigr\}\,d\lambda.
 \end{align}

The cases $|\Omega|=0$ or $\|u\|_{L^{a,\infty}}=0$ are immediate;
assume otherwise and set
\[
 \lambda_0
 :=\|u\|_{L^{a,\infty}}\,|\Omega|^{-1/a}.
\]
Splitting the integral \eqref{10011043} at $\lambda_0$, we obtain
\[
 \begin{aligned}
 \int_\Omega|u|
 &\leqslant
 |\Omega|\lambda_0
 +\|u\|_{L^{a,\infty}}^a
 \int_{\lambda_0}^{\infty}\lambda^{-a}\,d\lambda \\
 &=\|u\|_{L^{a,\infty}}|\Omega|^{1-1/a}
 +\frac1{a-1}
 \|u\|_{L^{a,\infty}}|\Omega|^{1-1/a} \\
 &=\frac{a}{a-1}
 \|u\|_{L^{a,\infty}}|\Omega|^{1-1/a}.
 \end{aligned}
\]
This proves the claimed estimate.
\end{proof}

Define the velocity-reconstructed pressure by
\begin{align}
 P&:=\sum_{i,j=1}^3\mathcal R_i\mathcal R_j(v_iv_j).
 \label{eq:entry-pressure-definition}
\end{align}
This is the standard Riesz-transform reconstruction of the pressure
in the whole space. Taking the divergence of the equation and using
$\operatorname{div}v=0$ gives
\[
 -\Delta p
 =\sum_{i,j=1}^3\partial_i\partial_j(v_iv_j),
\]
while \eqref{eq:riesz-poisson} gives the same equation for $P$.
Consequently $p-P$ is harmonic; the estimates below will show that it
is in fact constant. We shall refer to $P$ as the canonical pressure associated with $v$.
\begin{lemma}[The global entry estimates]
\label{lem:entry}
Let $(v,p)$ be a smooth solution of
\begin{equation}
 -\Delta v+(v\cdot\nabla)v+\nabla p=0,
 \qquad \operatorname{div}v=0
 \quad\hbox{in }\mathbb R^3,
 \qquad \lim_{|x|\to\infty}v(x)=0.
 \label{eq:entry-ns}
\end{equation}
Suppose that $\omega\in L^{s,\infty}$ for some
$9/5\le s<2$, and set
\begin{align}
 q&:=\frac{3s}{3-s},\qquad r:=\frac{3s}{6-s}.
 \label{eq:entry-exponents}
\end{align}
Then
\begin{equation}
 v\in L^{q,\infty}\cap L^6\cap L^\infty,
 \qquad
 \nabla v\in L^{s,\infty}\cap L^2,
 \qquad
 \omega\in L^{s,\infty}\cap L^{q/2,\infty}\cap L^2,
 \label{eq:entry-velocity}
\end{equation}
and
\begin{align}
 D&:=\int_{\mathbb R^3}|\nabla v|^2
   =\int_{\mathbb R^3}|\omega|^2<\infty.
 \label{eq:entry-dirichlet}
\end{align}
The canonical pressure $P$ defined in
\eqref{eq:entry-pressure-definition} (see also
Chae~\cite[proof of~(2.29), preceding~(2.30)]{Chae2026}) satisfies
\begin{equation}
 p-P\ \hbox{is constant},\qquad
 P\in L^3\cap L^{q/2,\infty},\qquad
 \nabla P\in L^2\cap L^{r,\infty}.
 \label{eq:entry-pressure-spaces}
\end{equation}
Moreover, every derivative of $v$ and $P$ is bounded and tends uniformly
to zero at infinity.  In particular, $\omega\in L^\infty$ and
$\omega(x)\to0$ as $|x|\to\infty$.
\end{lemma}

\begin{proof}
\noindent\textbf{Step 1. Velocity representation and weak estimates.}
Smoothness and the far-field condition in \eqref{eq:entry-ns} imply
$v\in L^\infty$.  The exponents defined in
\eqref{eq:entry-exponents} satisfy
\begin{equation}
 \frac92\le q<6,\qquad \frac97\le r<\frac32,\qquad
 \frac1q=\frac1s-\frac13,\qquad
 \frac1r=\frac1q+\frac1s,\qquad
 \frac1r-\frac13=\frac2q.
 \label{eq:entry-exponent-relations}
\end{equation}
In particular, both $s$ and $r$ are in the range $1<a<3$ required
by \eqref{eq:entry-potential-bounds}.

We first specify the realization of $\mathcal B$ used throughout
the proof.  For a smooth input $u\in L^{a,\infty}$, $1<a<3$,
the integral in \eqref{eq:biot-savart-definition} converges
absolutely near $y=x$, since $u$ is locally bounded and
$|x-y|^{-2}$ is locally integrable.  For its tail, applying
\eqref{eq:entry-weak-average} with exponent $a$ to the annuli
$2^kR<|x-y|<2^{k+1}R$ gives
\begin{align*}
 \int_{|x-y|>R}\frac{|u(y)|}{|x-y|^2}\,dy
 &\le C_a\|u\|_{L^{a,\infty}}
          \sum_{k=0}^{\infty}(2^kR)^{1-3/a}\\
 &\le C_a\|u\|_{L^{a,\infty}}R^{1-3/a}.
\end{align*}
The series converges because $a<3$.  Thus this integral represents
the distributional potential in \eqref{eq:entry-potential-bounds}.

By $\operatorname{div}v=0$ in \eqref{eq:entry-ns} and the definition
\eqref{eq:entry-vorticity-definition}, we have
$-\Delta v=\operatorname{curl}\omega$.
The definition \eqref{eq:biot-savart-definition} gives
$-\Delta\mathcal B\omega=\operatorname{curl}\omega$ in distributions.
Hence $v-\mathcal B\omega$ has an entire harmonic representative
by Weyl's lemma.  Apply \eqref{eq:entry-potential-bounds} with
$a=s$ and $u=\omega$.  Its output exponent is
$3s/(3-s)=q$, so $\mathcal B\omega\in L^{q,\infty}$.
For fixed $x$, the mean-value property and
\eqref{eq:entry-weak-average}, now with $a=q$ and
$\Omega=B_R(x)$, give
\[
 |(v-\mathcal B\omega)(x)|
 \le \frac{1}{|B_R|}\int_{B_R(x)}|v|
       +C R^{-3/q}\|\mathcal B\omega\|_{L^{q,\infty}}.
\]
The velocity average tends to zero: for any $\varepsilon>0$,
the far-field condition makes $|v|\le\varepsilon$ outside a fixed
ball, whose contribution to the average is $O(R^{-3})$.
Letting $R\to\infty$ and then $\varepsilon\to0$ removes the
harmonic term.  Applying both bounds in
\eqref{eq:entry-potential-bounds} with the same input now yields
\begin{equation}
 v=\mathcal B\omega,\qquad
 \|v\|_{L^{q,\infty}}+\|\nabla v\|_{L^{s,\infty}}
 \le C_s\|\omega\|_{L^{s,\infty}}.
 \label{eq:entry-first-potential}
\end{equation}

\smallskip
\noindent\textbf{Step 2. Improved vorticity integrability.}
Taking the curl of \eqref{eq:entry-ns} and using
$(v\cdot\nabla)v=\nabla(|v|^2/2)-v\times\omega$ gives
\begin{equation}
 -\Delta\omega=\operatorname{curl}(v\times\omega).
 \label{eq:entry-curl-equation}
\end{equation}
The Lorentz product inequality of O'Neil~\cite{ONeil1963},
with $1/r=1/q+1/s$ from \eqref{eq:entry-exponent-relations},
and the velocity bound in \eqref{eq:entry-first-potential} imply
\begin{align}
 \|v\times\omega\|_{L^{r,\infty}}
 \le C_s\|v\|_{L^{q,\infty}}\|\omega\|_{L^{s,\infty}}
 \le C_s\|\omega\|_{L^{s,\infty}}^2.\label{10011414}
 \end{align}
We may therefore apply \eqref{eq:entry-potential-bounds}
with $a=r$ and $u=v\times\omega$.  The relation
$1/r-1/3=2/q$ in \eqref{eq:entry-exponent-relations} gives
$3r/(3-r)=q/2$, and hence
$\mathcal B(v\times\omega)\in L^{q/2,\infty}$.

By definition \eqref{eq:biot-savart-definition},  $-\Delta\mathcal Bu=\operatorname{curl}u$.
Equation~\eqref{eq:entry-curl-equation} and
\eqref{eq:biot-savart-definition} imply that
$\omega-\mathcal B(v\times\omega)$ is entire harmonic.  Apply its
mean-value property on $B_R(x)$ and
\eqref{eq:entry-weak-average} to the two summands, with exponents
$s$ and $q/2$, respectively:
\[
 |\omega-\mathcal B(v\times\omega)|(x)
 \le C R^{-3/s}\|\omega\|_{L^{s,\infty}}
      +C R^{-6/q}\|\mathcal B(v\times\omega)\|_{L^{q/2,\infty}}
 \longrightarrow0.
\]
Thus $\omega=\mathcal B(v\times\omega)$. By \eqref{10011414} and
\eqref{eq:entry-potential-bounds}, we obtain
\begin{equation}
 \omega=\mathcal B(v\times\omega),\qquad
 \|\omega\|_{L^{q/2,\infty}}
 \le C_s\|\omega\|_{L^{s,\infty}}^2.
 \label{eq:entry-second-potential}
\end{equation}
The representations \eqref{eq:entry-first-potential} and
\eqref{eq:entry-second-potential} extend the strong-space formulae
of Chae--Wolf~\cite[Corollary~2.1]{ChaeWolf2016} to the present
weak-space setting. 

\smallskip
\noindent\textbf{Step 3. Finite energy and the div--curl identity.}
By \eqref{eq:entry-exponent-relations}, $s<2<q/2$.
In the below distribution-function formula for $\|\omega\|_2^2$,
use the weak-norm definition \eqref{eq:entry-weak-vorticity-norm}
with the assumed weak-$L^s$ bound on $0<\lambda<1$ and
the weak-$L^{q/2}$ bound in \eqref{eq:entry-second-potential}
on $\lambda>1$.  This gives
\begin{align}
 \int|\omega|^2
 &=2\int_0^\infty\lambda
             |\{|\omega|>\lambda\}|\,d\lambda \notag\\
 &\le 2\|\omega\|_{L^{s,\infty}}^s
               \int_0^1\lambda^{1-s}\,d\lambda
      +2\|\omega\|_{L^{q/2,\infty}}^{q/2}
               \int_1^\infty\lambda^{1-q/2}\,d\lambda \notag\\
 &=\frac{2\|\omega\|_{L^{s,\infty}}^s}{2-s}
       +\frac{2\|\omega\|_{L^{q/2,\infty}}^{q/2}}{q/2-2}
 <\infty.
 \label{eq:entry-vorticity-L2}
\end{align}
The first integral is precisely where the strict inequality $s<2$
is needed.  Since $\omega\in L^2$ and
$v=\mathcal B\omega$ by \eqref{eq:entry-first-potential},
the strong potential estimate
\eqref{eq:entry-strong-potential-bounds} with $a=2$ gives
\begin{equation}
 \|v\|_{L^6}
 +\|\nabla v\|_{L^2}
 \leqslant C\|\omega\|_{L^2}<\infty.
 \label{eq:entry-strong-velocity}
\end{equation}
In particular, $v$ has finite Dirichlet integral.  Combining
\eqref{eq:entry-first-potential}, \eqref{eq:entry-second-potential},
and \eqref{eq:entry-strong-velocity} proves all the assertions in
\eqref{eq:entry-velocity}.

The standard div--curl
identity on $\mathbb R^3$,
\[
 \int_{\mathbb R^3}|\nabla v|^2
 =
 \int_{\mathbb R^3}|\operatorname{curl}v|^2
 +
 \int_{\mathbb R^3}|\operatorname{div}v|^2,
\]
which is justified here by
$v\in L^6$ and $\nabla v\in L^2$, therefore gives
\[
 \int_{\mathbb R^3}|\nabla v|^2
 =
 \int_{\mathbb R^3}|\omega|^2.
\]
Because $\omega=\operatorname{curl}v$ by
\eqref{eq:entry-vorticity-definition}, the last equality is exactly
\eqref{eq:entry-dirichlet}.

\smallskip
\noindent\textbf{Step 4. Canonical pressure estimates.}
The bounds \eqref{eq:entry-first-potential} and
\eqref{eq:entry-strong-velocity} give $v\in L^{q,\infty}\cap L^6$.
Ordinary and Lorentz H\"older inequalities consequently yield
\[
 \|v_iv_j\|_{L^3}\le\|v\|_{L^6}^2,\qquad
 \|v_iv_j\|_{L^{q/2,\infty}}
 \le C_q\|v\|_{L^{q,\infty}}^2.
\]
Apply \eqref{eq:riesz-boundedness} to each product $v_iv_j$,
with exponents $3$ and $q/2$.  Summing over $i,j$ in
\eqref{eq:entry-pressure-definition} proves
$P\in L^3\cap L^{q/2,\infty}$.

For the gradient, the product rule and
\eqref{eq:riesz-commutation}, with $f=v_iv_j$, give in distributions
\begin{equation}
 \partial_\ell P
 =\sum_{i,j=1}^3\mathcal R_i\mathcal R_j
       \bigl[(\partial_\ell v_i)v_j
                    +v_i(\partial_\ell v_j)\bigr].
 \label{eq:entry-pressure-derivative}
\end{equation}
The products on the right lie in $L^2$ by $v\in L^\infty$ and
\eqref{eq:entry-strong-velocity}.  They also lie in $L^{r,\infty}$
by \eqref{eq:entry-first-potential} and Lorentz H\"older with
$1/r=1/q+1/s$ from \eqref{eq:entry-exponent-relations}.
Applying \eqref{eq:riesz-boundedness}, now with
exponents $2$ and $r$, to \eqref{eq:entry-pressure-derivative}
therefore gives
\begin{align}
 \|\nabla P\|_{L^2}
 &\le C\|v\|_{L^\infty}\|\nabla v\|_{L^2}, \notag\\
 \|\nabla P\|_{L^{r,\infty}}
 &\le C_s\|v\|_{L^{q,\infty}}\|\nabla v\|_{L^{s,\infty}}.
 \label{eq:entry-pressure-gradient-bounds}
\end{align}

\smallskip
\noindent\textbf{Step 5. Removal of the harmonic pressure ambiguity.}
Taking the divergence of \eqref{eq:entry-ns}, and applying
\eqref{eq:riesz-poisson} to $f=v_iv_j$ in
\eqref{eq:entry-pressure-definition}, gives
\[
 -\Delta p=\sum_{i,j=1}^3\partial_i\partial_j(v_iv_j)
           =-\Delta P.
\]
Since
\[
 \Delta(p-P)=0
 \quad\text{in }\mathcal D'(\mathbb R^3),
\]
Weyl's lemma implies that $p-P$ agrees almost everywhere with a
smooth harmonic function on $\mathbb R^3$. For each $\ell\in\{1,2,3\}$, define
\begin{align}
 H_\ell&:=\partial_\ell(p-P).
 \label{eq:pressure-harmonic-gradient}
\end{align}
Each $H_\ell$ is harmonic, and the momentum equation in
\eqref{eq:entry-ns} gives
\begin{equation}
 \nabla(p-P)=\Delta v-(v\cdot\nabla)v-\nabla P
 \quad\hbox{in }\mathcal D'(\mathbb R^3).
 \label{eq:entry-pressure-difference}
\end{equation}
Choose a nonnegative $\varphi\in C_c^\infty((0,1))$ normalized by
$4\pi\int_0^\infty r^2\varphi(r)\,dr=1$, and set
\begin{align}
 \varphi_R(y)&:=R^{-3}\varphi(|y|/R).
 \label{eq:pressure-test-function}
\end{align}
The kernel $\varphi_R$ is smooth even at the origin, since
$\varphi$ vanishes near zero. Polar coordinates
centered at $x$, and the spherical mean-value property of the harmonic
function $H_\ell$ give
\[
 \begin{aligned}
 \int_{\mathbb R^3}H_\ell(y)\varphi_R(x-y)\,dy
 &=
 \int_0^\infty
 R^{-3}\varphi(r/R)
 \int_{|y-x|=r}H_\ell(y)\,dS_y\,dr\\
 &=
 H_\ell(x)
 \int_0^\infty4\pi r^2R^{-3}\varphi(r/R)\,dr\\
 &=
 H_\ell(x),
 \end{aligned}
\]
because $\int_{\mathbb R^3}\varphi_R=1$.  Hence
\[
 H_\ell(x)
 =
 \langle H_\ell,\varphi_R(x-\cdot)\rangle.
\]

Pair \eqref{eq:entry-pressure-difference} with
$\varphi_R(x-\cdot)$, transfer one derivative from $\Delta v$ to
the test function, and use Cauchy--Schwarz.  The scaling in
\eqref{eq:pressure-test-function} gives
$\|\nabla\varphi_R\|_{L^2(\mathbb R^3)}\le CR^{-5/2}$ and
$\|\varphi_R\|_{L^2(\mathbb R^3)}\le CR^{-3/2}$, so
\begin{align*}
 |\nabla(p-P)(x)|
 &\le C R^{-5/2}\|\nabla v\|_2\\
 &\quad+C R^{-3/2}
       \bigl(\|(v\cdot\nabla)v\|_2+\|\nabla P\|_2\bigr)
 \longrightarrow0.
\end{align*}
All norms are finite by \eqref{eq:entry-strong-velocity},
\eqref{eq:entry-pressure-gradient-bounds}, and
$\|(v\cdot\nabla)v\|_2\le\|v\|_\infty\|\nabla v\|_2$.
Hence $p-P$ is constant, completing \eqref{eq:entry-pressure-spaces};
in particular, $P$ is smooth.

\smallskip
\noindent\textbf{Step 6. Boundedness and decay of derivatives.}
Apply the exterior-domain decay theorem
\cite[Theorem~X.5.1]{Galdi2011} to the restriction of
$(v,p)$ to the exterior of a fixed ball, with the smooth trace of
$v$ as boundary data.  Its hypotheses hold: the restricted solution
is smooth and divergence free, its Dirichlet integral is finite by
\eqref{eq:entry-dirichlet}, its far-field value is zero by
\eqref{eq:entry-ns}, and its body force is zero.  The boundary flux
condition follows from the divergence theorem:
\[
 \int_{\partial B_R}v\cdot n\,dS
 =\int_{B_R}\operatorname{div}v\,dx=0.
\]
The theorem yields a constant $p_\infty$ such that, for every fixed
multi-index $\alpha$, both $\partial^\alpha v$ and
$\partial^\alpha(p-p_\infty)$ tend uniformly to zero at infinity,
including the order-zero limits.

Since $p-P$ is constant and $P=p-(p-P)$, the limit of $P$ at infinity
is the constant $p_\infty-(p-P)$.  Because $P\in L^3$, this constant
must vanish.  Hence
\begin{equation}
 p-P=p_\infty,
 \qquad
 P=p-p_\infty\to0
 \label{eq:entry-pressure-normalization}
\end{equation}
uniformly at infinity.  For every multi-index $\alpha$ with
$|\alpha|\ge1$, the identity $P=p-p_\infty$ gives
$\partial^\alpha P=\partial^\alpha(p-p_\infty)$, so the derivative
decay supplied by the exterior-domain theorem applies to $P$ as well
as to $v$.  Their order-zero decay is given by
\eqref{eq:entry-pressure-normalization} for $P$ and by
\eqref{eq:far-field} for $v$.  Smoothness bounds all these
derivatives on a fixed ball, and
$\omega=\operatorname{curl}v$ then has the asserted boundedness and
decay.
\end{proof}

\begin{remark}[Equivalent formulations of finite energy]
\label{rem:div-curl-equivalence}
The equivalence stated in the introduction is independent of the
Navier--Stokes equation.  Indeed, if $\nabla v\in L^2$, the
homogeneous Sobolev inequality and \eqref{eq:far-field} give
$v\in L^6$, and the standard div--curl identity for
$\operatorname{div}v=0$ yields
\[
 \|\nabla v\|_{L^2}^2=\|\omega\|_{L^2}^2.
\]
Conversely, if $\omega\in L^2$, then
\eqref{eq:entry-strong-potential-bounds} gives
\[
 \mathcal B\omega\in L^6,
 \qquad
 \nabla\mathcal B\omega\in L^2.
\]
Since $-\Delta v=\operatorname{curl}\omega$,
the difference $v-\mathcal B\omega$ is entire harmonic.  Its
mean-value property on $B_R(x)$, \eqref{eq:far-field}, and the
$L^6$ conclusion of \eqref{eq:entry-strong-potential-bounds} give
\[
 |(v-\mathcal B\omega)(x)|
 \le \frac1{|B_R|}\int_{B_R(x)}|v|
      +CR^{-1/2}\|\mathcal B\omega\|_{L^6}\longrightarrow0.
\]
Indeed, the first average tends to zero by \eqref{eq:far-field}, while
the second term tends to zero because $\mathcal B\omega\in L^6$.
Hence $v=\mathcal B\omega$ and $\nabla v\in L^2$, after which the
div--curl identity gives \eqref{eq:intro-div-curl}.

Thus, in the smooth divergence-free zero-far-field class,
$\nabla v\in L^2$ and $\omega\in L^2$ are equivalent.  This
equivalence alone gives no Liouville conclusion; vanishing under
finite energy is precisely the classical finite-energy Liouville
problem.
\end{remark}

\begin{remark}[Use of the exterior-domain decay theorem]
\label{rem:entry-galdi}
Theorem~X.5.1 of Galdi~\cite{Galdi2011} is used only
after finite Dirichlet integral has been proved from the weak
vorticity hypothesis.  It supplies boundedness and uniform decay of
all fixed derivatives; it contributes no quantitative power of
$|x|$ and no Liouville conclusion.
\end{remark}

From now on, we subtract the constant in
\eqref{eq:entry-pressure-spaces} and normalize the pressure so that
$p=P$.  Define
\begin{align}
 h&:=-p-\frac{|v|^2}{2}.
 \label{eq:entry-bernoulli-definition}
\end{align}

\begin{lemma}[The positive Bernoulli variable]
\label{lem:entry-bernoulli}
Under the hypotheses of Lemma~\ref{lem:entry}, either $v=0$, or
\begin{equation}
 h>0,\qquad -p=h+\frac{|v|^2}{2}\ge h>0
 \quad\hbox{throughout }\mathbb R^3.
 \label{eq:entry-positive}
\end{equation}
In the latter case $D>0$.  The quantities $h$, $-p$, and
$|v|^2/2$ satisfy
\begin{equation}
 \Delta h=v\cdot\nabla h-|\omega|^2,\qquad
 \Delta \frac{|v|^2}{2}=|\nabla v|^2-v\cdot\nabla h,\qquad
 -\Delta p=|\nabla v|^2-|\omega|^2.
 \label{eq:entry-scalar-equations}
\end{equation}
Moreover,
\begin{gather}
 \frac{|v|^2}{2},h,-p\in L^3\cap L^{q/2,\infty}\cap L^\infty,
 \qquad |v(x)|^2,h(x),p(x)\longrightarrow0 \quad\hbox{ as } |x| \rightarrow \infty.
 \label{eq:entry-scalar-spaces}\\
 \nabla \frac{|v|^2}{2},\nabla h,\nabla p\in L^2\cap L^{r,\infty}.
 \label{eq:entry-scalar-gradient-spaces}
\end{gather}
Finally, define the deformation tensor by
\begin{equation}
 \begin{aligned}
 E&:=\frac{\nabla v+\nabla v^{\mathsf T}}2
       =(E_{ij})_{1\le i,j\le3},\\
 E_{ij}&:=\frac{\partial_i v_j+\partial_j v_i}{2},
       \qquad 1\le i,j\le3.
 \end{aligned}
 \label{eq:entry-strain-definition}
\end{equation}
Then
\begin{equation}
 \operatorname{tr}E=0,\qquad
 \int|E|^2=\frac D2.
 \label{eq:entry-strain-energy}
\end{equation}
\end{lemma}

\begin{proof}
Taking the scalar product of the first equation in
\eqref{eq:entry-ns} with $v$, and then taking the divergence of that
equation, give
\eqref{eq:entry-scalar-equations}.  In particular,
\[
 (\Delta-v\cdot\nabla)h=-|\omega|^2,
 \qquad h(x)\longrightarrow0
 \quad\text{as }|x|\to\infty.
\]
The standard Bernoulli maximum-principle argument yields either
$h\equiv0$ or $h>0$ throughout $\mathbb R^3$; see
Chae--Wolf~\cite[proof of Theorem~1.2]{ChaeWolf2016}
or Chae~\cite[formulae~(2.1)--(2.4)]{Chae2026}.
If $h\equiv0$, then $\omega=0$, so
\eqref{eq:entry-dirichlet} and \eqref{eq:far-field} imply $v=0$.
Thus, whenever $v\ne0$, we have $h>0$ and $D>0$, and
$-p=h+|v|^2/2$ gives \eqref{eq:entry-positive}.

With $p=P$, the velocity bounds in \eqref{eq:entry-velocity} and
\eqref{eq:entry-strong-velocity}, together with the pressure bounds in
\eqref{eq:entry-pressure-spaces}, imply
$|v|^2/2,p\in L^3\cap L^{q/2,\infty}\cap L^\infty$.  The definition
\eqref{eq:entry-bernoulli-definition} then gives the same integrability
for $h$.  The decay statements in \eqref{eq:entry-scalar-spaces}
follow from the uniform decay conclusion of Lemma~\ref{lem:entry} and
\eqref{eq:entry-bernoulli-definition}.  For the gradient assertions,
\eqref{eq:entry-velocity}, \eqref{eq:entry-strong-velocity}, and
\eqref{eq:entry-pressure-spaces}, combined with the strong and weak
H\"older inequalities and the exponent identity in
\eqref{eq:entry-exponent-relations}, give
\[
 \left|\nabla\frac{|v|^2}{2}\right|
 \le |v|\,|\nabla v|,
 \qquad
 \frac1r=\frac1q+\frac1s.
\]

Finally, $\operatorname{tr}E=\operatorname{div}v=0$, and the
orthogonal decomposition of $\nabla v$ into its symmetric and
skew-symmetric parts gives
\begin{equation}
 |E|^2=|\nabla v|^2-\frac12|\omega|^2.
 \label{eq:entry-strain-decomposition}
\end{equation}
Integrating and using \eqref{eq:entry-dirichlet} yields
\[
 \int_{\mathbb R^3}|E|^2
 =D-\frac D2=\frac D2,
\]
which proves \eqref{eq:entry-strain-energy}.
\end{proof}

\section{The Bernoulli--vorticity quotient identity}
\label{sec:ratio}

We work on the nonzero branch of Lemma~\ref{lem:entry-bernoulli},
so $D>0$ and $h>0$. Retain $h$ and $E$ from
Section~\ref{sec:entry}, and define
\begin{align}
 L&:=\Delta-v\cdot\nabla.
 \label{eq:ratio-operator-definition}
\end{align}
Taking the curl of the first equation in \eqref{eq:entry-ns} gives the
vorticity equation.  Together with
\eqref{eq:entry-scalar-equations}, it reads
\begin{equation}
 Lh=-|\omega|^2,\qquad L\omega=-(\omega\cdot\nabla)v,
 \qquad \Delta \frac{|v|^2}{2}+v\cdot\nabla h=|\nabla v|^2.
 \label{eq:ratio-basic-equations}
\end{equation}
The pointwise decomposition \eqref{eq:entry-strain-decomposition},
the trace identity in \eqref{eq:entry-strain-energy}, and the integral
identity obtained there from \eqref{eq:entry-dirichlet} read
\begin{equation}
 |\nabla v|^2=|E|^2+\frac{|\omega|^2}{2},
 \qquad \operatorname{tr}E=0,
 \qquad \int |E|^2=\frac D2.
 \label{eq:ratio-strain-energy}
\end{equation}
Since $h>0$, the quotient $|\omega|^2/h$ is smooth on
$\mathbb R^3$, but its integrability at infinity is not yet known
because $h(x)\to0$. The following lemma uses spatial cutoffs to
derive the global bounds from its differential identity.

\begin{lemma}[The global ratio budget]
\label{lem:ratio-budget}
Define
\begin{align}
 F&:=\frac{|\omega|^2}{h},\qquad
 W:=\frac\omega h,\qquad
 \mathcal G:=\frac{\omega\otimes\omega}{h}-\frac F3 I.
 \label{eq:ratio-quantities}
\end{align}
Then
\begin{equation}
 LF=F^2-2\operatorname{tr}\!\left(
       E\mathcal G\right)
       +2h|\nabla W|^2
 \label{eq:ratio-differential}
\end{equation}
pointwise in \(\mathbb R^3\).  Moreover,
\begin{equation}
 \int_{\mathbb R^3}F^2\,dx\le\frac{4D}{3},\qquad
 \int_{\mathbb R^3}h|\nabla W|^2\,dx\le\frac D6.
 \label{eq:ratio-global-bounds}
\end{equation}
\end{lemma}

\begin{proof}
The first two identities in \eqref{eq:ratio-basic-equations} give
$Lh=-|\omega|^2$ and $L\omega=-(\omega\cdot\nabla)v$.  The
skew-symmetric part of $\nabla v$ annihilates $\omega$, so
$(\omega\cdot\nabla)v=E\omega$. Applying the product and
quotient rules to $F=|\omega|^2/h$ therefore produces
\[
 LF=F^2-\frac2h\omega^{\mathsf T}E\omega
       +2h|\nabla W|^2.
\]
Since $\operatorname{tr}E=0$ by
\eqref{eq:ratio-strain-energy}, the
definition of $\mathcal G$ in \eqref{eq:ratio-quantities} rewrites the
middle term as
$-2\operatorname{tr}\!\left(
E\mathcal G\right)$. This proves
\eqref{eq:ratio-differential}.  The same definition
\eqref{eq:ratio-quantities} also gives
$|\mathcal G|^2=2F^2/3$.

To obtain the global bound without assuming $F\in L^2$, choose
$\chi\in C_c^\infty(B_2)$ with $0\le\chi\le1$ and
$\chi=1$ on $B_1$, and set $\chi_R(x):=\chi(x/R)$.
Define
\begin{align}
 A_R&:=B_{2R}\setminus B_R,\notag\\
 X_R&:=\int\chi_R^4F^2,\qquad
 Y_R:=\int\chi_R^4h|\nabla W|^2.
 \label{eq:ratio-cutoff-quantities}
\end{align}
These integrals are finite for each $R$, since $h>0$.
Testing \eqref{eq:ratio-differential} against $\chi_R^4$
and integrating by parts, using $\operatorname{div}v=0$, gives
\begin{align}
 X_R+2Y_R
 &\le 2\int\chi_R^4\operatorname{tr}\!\left(
          E\mathcal G\right)
   +\int F\bigl(\Delta\chi_R^4+v\cdot\nabla\chi_R^4\bigr).
 \label{eq:ratio-cutoff-identity}
\end{align}

By Cauchy--Schwarz and H\"older, using $|A_R|\le CR^3$,
\begin{align}
 \left|\int F\Delta\chi_R^4\right|
 &\le CR^{-2}|A_R|^{1/2}X_R^{1/2}
 \le CR^{-1/2}X_R^{1/2},
 \label{eq:ratio-cutoff-diffusion}\\
 \left|\int Fv\cdot\nabla\chi_R^4\right|
 &\le CR^{-1}\|v\|_{L^2(A_R)}X_R^{1/2}\notag\\
 &\le CR^{-1}|A_R|^{1/3}
          \|v\|_{L^6(A_R)}X_R^{1/2}\notag\\
 &\le C\|v\|_{L^6(A_R)}X_R^{1/2}.
 \label{eq:ratio-cutoff-drift}
\end{align}
Moreover, $|\mathcal G|^2=2F^2/3$ and
$\int|E|^2=D/2$ yield
\begin{align}
 2\left|\int\chi_R^4
       \operatorname{tr}\!\left(
          E\mathcal G\right)\right|
 &\le 2\sqrt{\frac23}\,\|E\|_{L^2}X_R^{1/2}
 =2\sqrt{\frac D3}\,X_R^{1/2}.
 \label{eq:ratio-strain-bound}
\end{align}
Substituting \eqref{eq:ratio-cutoff-diffusion},
\eqref{eq:ratio-cutoff-drift}, and \eqref{eq:ratio-strain-bound} into
\eqref{eq:ratio-cutoff-identity}, we obtain
\begin{equation}\label{eq:ratio-two-budgets}
 X_R+2Y_R
 \le\left(2\sqrt{\frac D3}+\varepsilon_R\right)X_R^{1/2},
 \qquad
 \varepsilon_R:=C\bigl(R^{-1/2}+\|v\|_{L^6(A_R)}\bigr).
\end{equation}
The velocity estimate $v\in L^6$ in \eqref{eq:entry-strong-velocity}
implies $\varepsilon_R\to0$. Dropping $Y_R\ge0$ from
\eqref{eq:ratio-two-budgets} gives
\[
 X_R^{1/2}\le 2\sqrt{\frac D3}+\varepsilon_R.
\]
Hence Fatou's lemma yields
\[
 \int_{\mathbb R^3}F^2
 \le \liminf_{R\to\infty}X_R
 \le \frac{4D}{3}.
\]
To retain the quotient-gradient term, set
$a_R:=2\sqrt{D/3}+\varepsilon_R$. By
$a_Rx-x^2\le a_R^2/4$ and \eqref{eq:ratio-two-budgets},
\[
 2Y_R\le a_RX_R^{1/2}-X_R\le\frac{a_R^2}{4}.
\]
Another application of Fatou's lemma therefore gives
\[
 \int_{\mathbb R^3}h|\nabla W|^2
 \le\liminf_{R\to\infty}Y_R
 \le\frac18\left(2\sqrt{\frac D3}\right)^2
 =\frac D6.
\]
This proves both estimates in \eqref{eq:ratio-global-bounds}.
\end{proof}

Combining \eqref{eq:ratio-differential} with
\eqref{eq:ratio-strain-energy} and the identity
$|\mathcal G|^2=2F^2/3$ obtained from
\eqref{eq:ratio-quantities}, we complete the square in the
strain term to obtain
\begin{equation}\label{eq:ratio-completed-square}
 |\nabla v|^2+LF
 =|E-\mathcal G|^2+\frac{|\omega|^2}{2}
       +\frac{F^2}{3}+2h|\nabla W|^2.
\end{equation}
All terms on the right are nonnegative.

To obtain estimates on the level sets of $h$, we now allow
the denominator of $F$ to vary.  For a positive function
$\psi\in C^2((0,\infty))$, define
\[
 F_\psi:=\frac{|\omega|^2}{\psi(h)}.
\]

Applying the product and quotient rules to the first two identities in
\eqref{eq:ratio-basic-equations}, and then using
\eqref{eq:ratio-strain-energy} to complete the same trace-free strain
square as in \eqref{eq:ratio-completed-square}, we obtain
\begin{align}
& |\nabla v|^2+LF_\psi \notag\\
 &=\left|E-\frac{\omega\otimes\omega}{\psi(h)}
             +\frac{F_\psi}{3}I\right|^2
       +\frac{|\omega|^2}{2} +\left(\psi'(h)-\frac23\right)F_\psi^2\notag\\
 &\quad
       +2\psi(h)\left|
          \nabla\left(\frac{\omega}{\psi(h)}\right)
        \right|^2 -\frac{\psi''(h)|\omega|^2}{\psi(h)^2}
          |\nabla h|^2.
 \label{eq:ratio-general-denominator}
\end{align}
Since $\operatorname{div}v=0$,
\eqref{eq:ratio-basic-equations} also gives
\[
 |\nabla v|^2+LF_\psi
 =
 \diver\left\{
   \nabla\left(\frac{|v|^2}{2}+F_\psi\right)
   +(h-F_\psi)v
 \right\}.
\]
Thus \eqref{eq:ratio-general-denominator} is a divergence
identity with a nonnegative right-hand side whenever
$\psi'\ge2/3$ and $\psi''\le0$.
Taking $\psi(r)=r$ recovers
\eqref{eq:ratio-completed-square}.

For a positive regular height $\tau$, the concave choice
\begin{equation}
 \psi_\tau(r):=r+\min(r,\tau)
 \label{eq:ratio-concave-denominator}
\end{equation}
has slope $2$ below $\tau$ and slope $1$ above it. Both slopes
preserve positivity of the volume terms in
\eqref{eq:ratio-general-denominator}. Its second derivative is the
negative unit point mass at $\tau$, so the term involving
$-\psi''(h)$ in \eqref{eq:ratio-general-denominator} becomes
\begin{equation}\label{eq:ratio-concave-interface}
 \frac{|\omega|^2}{4\tau^2}|\nabla h|\,dS\big|_{\{h=\tau\}}\ge0.
\end{equation}

In Section~\ref{sec:entropy}, the denominator
\eqref{eq:ratio-concave-denominator} is inserted into a lower-power,
pressure-weighted profile to prove the trace estimate
\eqref{eq:convex-entropy-trace}.  Its entropy inequality also recovers
the finiteness of the gradient integral already quantified in
\eqref{eq:ratio-global-bounds}.

\section{Pressure ratios on Bernoulli layers}
\label{sec:pressure-tail}

Continue on the nonzero branch with $9/5<s<2$. In addition to $F,W$
from \eqref{eq:ratio-quantities} and $E$ from
\eqref{eq:entry-strain-definition}, define the positive pressure
magnitude and its Bernoulli ratio by
\begin{equation}
 \Pi:=-p,\qquad f:=\frac{\Pi}{h}.
 \label{eq:positive-pressure-definition}
\end{equation}
The pressure equation in \eqref{eq:entry-scalar-equations},
the vorticity equation in \eqref{eq:ratio-basic-equations},
and the momentum equation give
\begin{align}
 \Delta\Pi&=|\nabla v|^2-|\omega|^2,
 \qquad L\omega=-E\omega,\notag\\
 \nabla h&=\curl\omega-v\times\omega.
 \label{eq:convex-pde}
\end{align}
Here $(\omega\cdot\nabla)v=E\omega$, since the antisymmetric
part of $\nabla v$ annihilates $\omega$.
The third identity in \eqref{eq:convex-pde} is the Bernoulli form of the momentum
equation; see also Lerner~\cite[formula~(3.1.2)]{Lerner2026}.
Lemma~\ref{lem:entry} gives boundedness and decay of $h$ and $\Pi$,
bounded vorticity, and the Lorentz estimates needed below.
Write
\begin{equation}
 \lambda:=\frac{2(2-s)}{4-s},\qquad
 \delta:=\frac{5s-9}{2(3-s)},
 \label{pm:parameters}
\end{equation}
and define
\begin{equation}
 \mathcal S_t:=\{t<h<2t\},\qquad
 d\mu_t:= t\mathbf1_{\mathcal S_t}|\nabla\log h|^2\dd.
 \label{eq:ratio-band-definitions}
\end{equation}
The constants $C$ below may depend on the fixed solution
and the fixed exponents, but are independent of the level
parameters $t,\tau$, the pressure cutoff $a$, and the
pressure-ratio threshold $K$ introduced below.

For each positive regular value $\tau$, the set $\{h>\tau\}$ is bounded.
Its outward unit normal is $-\nabla h/|\nabla h|$, so
\[
 \int_{\{h>\tau\}}\Delta h\dd=-\int_{h=\tau}|\nabla h|\,dS.
\]
Since $\operatorname{div}v=0$ and $h=\tau$ on
$\partial\{h>\tau\}$, the divergence theorem gives
\begin{align*}
 \int_{\{h>\tau\}}v\cdot\nabla h\dd
 &=\tau\int_{\partial\{h>\tau\}}v\cdot n\,dS\\
 &=\tau\int_{\{h>\tau\}}\operatorname{div}v\dd
 =0.
\end{align*}
Integrating $Lh=-|\omega|^2$ therefore gives
\[
 A(\tau):=\int_{h=\tau}|\nabla h|\,dS
       =\int_{\{h>\tau\}}|\omega|^2\dd,
 \qquad 0\le A(\tau)\le D,\qquad A(\tau)\longrightarrow D
 \quad(\tau\downarrow0).
\]
The limit $A(\tau)\to D$ follows because $h>0$ everywhere and
$\omega\in L^2$ by \eqref{eq:entry-dirichlet}.  We recall the weighted
coarea formula for the smooth function $h$:
for $0<a<b$ and any nonnegative measurable function $\varphi$,
\begin{align}
 \int_{\{a<h<b\}}\varphi(x)|\nabla h(x)|\,dx
 &=
 \int_a^b
 \left(\int_{\{h=\tau\}}\varphi(x)\,dS(x)\right)d\tau.
 \label{eq:weighted-coarea}
\end{align}
Here the surface integrals are understood for almost every
level $\tau$; see
Evans--Gariepy~\cite[Theorem~3.11]{EvansGariepy2015}.

Applying \eqref{eq:weighted-coarea} with $a=t$, $b=2t$,
and $\varphi=|\nabla h|/h^2$, and then substituting $\tau=tr$,
we obtain
\begin{align}
 \mu_t(\mathbb R^3)
 &=t\int_{\{t<h<2t\}}\frac{|\nabla h|^2}{h^2}\dd
   \notag\\
 &=t\int_t^{2t}\tau^{-2}
       \left(\int_{\{h=\tau\}}|\nabla h|\,dS\right)d\tau
   \notag\\
 &=t\int_t^{2t}\tau^{-2}A(\tau)\,d\tau
   \notag\\
 &=\int_1^2r^{-2}A(tr)\,dr
   \longrightarrow\frac D2.
 \label{eq:convex-mass}
\end{align}
The limit follows by dominated convergence from
$0\le A(\tau)\le D$, $A(\tau)\to D$ as $\tau\downarrow0$,
and $\int_1^2r^{-2}\,dr=1/2$.
In particular, $\mu_t(\mathbb R^3)\le D/2$ for every $t>0$. 

\begin{lemma}\label{lem:convex-tail}
For small $t>0$ and $K\ge4$,
\begin{equation}\label{eq:convex-tail}
 \mu_t(f>K)\le Ct^{-\delta}K^{-1-\delta}.
\end{equation}
Consequently, for every measurable $\Omega\subset\mathcal S_t$,
\begin{equation}\label{eq:convex-pressure-core}
 \int_\Omega f\,d\mu_t\le Ct^{-\delta/(1+\delta)}
                         \mu_t(\Omega)^{\delta/(1+\delta)}.
\end{equation}
\end{lemma}

\begin{proof}
Fix a pressure threshold $a>0$, independently of $\tau$ and $t$, and let
$H_a:=(\Pi-a)_+$. Its support is compact, since $\Pi$ tends to zero.
Initially suppose that $a$ is a regular value of $\Pi$. The first
identity in \eqref{eq:convex-pde} gives the distributional identity
\begin{equation}\label{eq:pressure-truncation}
 \Delta H_a=\mathbf1_{\{\Pi>a\}}
       \bigl(|\nabla v|^2-|\omega|^2\bigr)
       +|\nabla\Pi|\,dS\big|_{\Pi=a}.
\end{equation}
Integration on the pressure superlevel set, whose outward normal is
$-\nabla\Pi/|\nabla\Pi|$, also gives
\begin{equation}\label{eq:pressure-positive-budget}
 \int_{\{\Pi>a\}}|\nabla v|^2\dd
       +\int_{\{\Pi=a\}}|\nabla\Pi|\,dS
       =\int_{\{\Pi>a\}}|\omega|^2\dd\le D.
\end{equation}
In particular, the positive interface term must be kept together with
the volume integral of $|\nabla v|^2$.

We next derive the mixed Green identity used below. Define
\begin{align}
M_a(\tau):=\int_{h=\tau}H_a|\nabla h|\,dS,\\
 T_a(\tau):=\int_{\{h<\tau\}}H_av\cdot\nabla h\dd. \label{10021434}
 \end{align}
On the boundary of $\{h<\tau\}$, the outward normal is
$n_h:=\nabla h/|\nabla h|$. For a smooth approximation to $H_a$,
the divergence theorem gives
\begin{align}
 \int_{\{h<\tau\}}(H_a\Delta h-h\Delta H_a)
 &=M_a(\tau)-\tau\int_{h=\tau}\partial_{n_h}H_a\,dS,\label{10021359}\\
 \int_{\{h<\tau\}}\Delta H_a
 &=\int_{h=\tau}\partial_{n_h}H_a\,dS.\label{10021358}
\end{align}
The compact support of the pressure factor eliminates any boundary
term at infinity. Substituting the first
identity from
\eqref{eq:ratio-basic-equations} and \eqref{10021358} into \eqref{10021359}
yields
\begin{equation}\label{eq:pressure-mixed-green}
 M_a(\tau)=T_a(\tau)-\int_{\{h<\tau\}}H_a|\omega|^2\dd
                   +\int_{\{h<\tau\}}(\tau-h)\,\Delta H_a.
\end{equation}
Standard mollification justifies the passage to $H_a$, since the
Laplacian term is paired with the continuous weight $(\tau-h)_+$.
The identity holds for every positive regular value $\tau$ of $h$,
which suffices for the coarea argument below.
Since $0<\tau-h<\tau$ on $\{h<\tau\}$, \eqref{eq:pressure-truncation} and
\eqref{eq:pressure-positive-budget} imply
\[
 \int_{\{h<\tau\}}(\tau-h)\,\Delta H_a
 \le \tau\left(\int_{\{\Pi>a\}}|\nabla v|^2\dd+
                    \int_{\{\Pi=a\}}|\nabla\Pi|\,dS\right)\le \tau D.
\]
The term $-\int_{\{h<\tau\}}H_a|\omega|^2$ is also nonpositive.
Consequently,
\begin{align}\label{10021439}
	M_a(\tau)\le T_a(\tau)+\tau D.
\end{align}

To estimate the transport term, integrate the divergence of the
compactly supported vector field $(\tau-h)_+H_av$:
\[
 0=-\int_{\{h<\tau\}}H_av\cdot\nabla h\dd
   +\int_{\{h<\tau\}}(\tau-h)\mathbf1_{\{\Pi>a\}}
                         v\cdot\nabla\Pi\dd.
\]
Thus
\begin{align}
T_a(\tau)=\int_{\{h<\tau\}}(\tau-h)\mathbf1_{\{\Pi>a\}}
                                  v\cdot\nabla\Pi\dd.\label{10021435}
                                  \end{align}
The exponent relations in \eqref{eq:entry-exponent-relations}, the
velocity estimate \eqref{eq:entry-first-potential}, and the pressure
estimate \eqref{eq:entry-pressure-spaces} give
\[
 q=\frac{3s}{3-s},\qquad r=\frac{3s}{6-s},\qquad
 v\in L^{q,\infty},\quad \nabla\Pi\in L^{r,\infty},
 \quad \Pi\in L^{q/2,\infty}.
\]
Weak H\"older therefore gives $v\cdot\nabla\Pi\in L^{m,\infty}$,
where
\[
 m:=\frac{3s}{9-2s}>1,\qquad
 \frac1m=\frac1q+\frac1r.
\]
Since $a>0$, the pressure bound in
\eqref{eq:entry-pressure-spaces} and the definition
\eqref{eq:weak-norm} give
\[
 |\{\Pi>a\}|
 \le \|\Pi\|_{L^{q/2,\infty}}^{q/2}a^{-q/2}.
\]
Using the expression \eqref{10021435} and
$0<\tau-h<\tau$ on $\{h<\tau\}$, we apply
\eqref{eq:entry-weak-average} with exponent $m$ to
$v\cdot\nabla\Pi$ on $\{\Pi>a\}$. This yields
\begin{align}\label{10021438}
 |T_a(\tau)|
 &\le \tau\int_{\{\Pi>a\}}|v\cdot\nabla\Pi|\dd \notag\\
 &\le \frac{m}{m-1}\tau
       \|v\cdot\nabla\Pi\|_{L^{m,\infty}}
       |\{\Pi>a\}|^{1-1/m} \notag\\
 &\le C\tau a^{-\frac q2(1-\frac1m)}.
\end{align}
The exponent is
\[
 \frac q2\left(1-\frac1m\right)
 =\frac{3s}{2(3-s)}
       \left(1-\frac{9-2s}{3s}\right)
 =\frac{5s-9}{2(3-s)}
 =\delta.
\]
Combining \eqref{10021438} with
\eqref{10021439}, we obtain
\begin{equation}\label{eq:convex-Ma}
 |T_a(\tau)|\le C\tau a^{-\delta},
 \qquad
 M_a(\tau)\le C\tau a^{-\delta}+\tau D.
\end{equation}
The constants are independent of $a$ and $\tau$.

A standard smooth convex approximation of the positive-part function
in $H_a=(\Pi-a)_+$ allows us to pass to the limit and extend
\eqref{eq:convex-Ma} from regular pressure thresholds to every $a>0$.

Only now choose $a:=Kt/2$.  For $t<\tau<2t$, the condition $f>K$
on $\{h=\tau\}$ is equivalent to $\Pi>K\tau$. Since $a=Kt/2$,
this implies $\Pi>2a$ and hence $H_a=\Pi-a\ge a$.
Applying the weighted coarea formula \eqref{eq:weighted-coarea}
and then \eqref{eq:convex-Ma}, we obtain
\begin{align*}
 \mu_t(f>K)
 &=t\int_t^{2t}\tau^{-2}
      \left(\int_{\{h=\tau,\ \Pi>K\tau\}}
                    |\nabla h|\,dS\right)d\tau\\
 &\le\frac{t}{a}\int_t^{2t}\tau^{-2}
      \left(\int_{\{h=\tau\}}H_a|\nabla h|\,dS\right)d\tau\\
 &=\frac{t}{a}\int_t^{2t}\tau^{-2}M_a(\tau)\,d\tau\\
 &\le\frac{t}{a}\int_t^{2t}\tau^{-2}
                  \bigl(C\tau a^{-\delta}+\tau D\bigr)\,d\tau\\
 &=(\log2)\left(Ct a^{-1-\delta}+\frac{Dt}{a}\right).
\end{align*}
The desired estimate is trivial if $\mathcal S_t\cap\{f>K\}$
is empty. Otherwise, $\Pi>Kh>Kt$ at some point, so
$Kt<\|\Pi\|_\infty$. Substituting $a=Kt/2$ into the preceding
bound gives
\begin{align*}
 \mu_t(f>K)
 &\le C\bigl(t^{-\delta}K^{-1-\delta}+K^{-1}\bigr)\\
 &=Ct^{-\delta}K^{-1-\delta}\bigl(1+(Kt)^\delta\bigr)\\
 &\le Ct^{-\delta}K^{-1-\delta}.
\end{align*}
Thus \eqref{eq:convex-tail} holds for every $K>0$.

Finally, the weak-average estimate \eqref{eq:entry-weak-average}
applies equally to the measure $\mu_t$. Applying it with exponent
$1+\delta>1$ and using \eqref{eq:convex-tail}, we obtain
\[
 \int_\Omega f\,d\mu_t
 \le \frac{1+\delta}{\delta}
       (Ct^{-\delta})^{1/(1+\delta)}
       \mu_t(\Omega)^{\delta/(1+\delta)}
 \le C t^{-\delta/(1+\delta)}
       \mu_t(\Omega)^{\delta/(1+\delta)}.
\]
This proves \eqref{eq:convex-pressure-core}.
\end{proof}

\section{Convex entropy and weighted vorticity traces}
\label{sec:entropy}

We first obtain the inverse powers of the positive pressure magnitude
that will control the error terms in the entropy calculation. Fix
$0<\kappa<\lambda$, where
$\lambda=2(2-s)/(4-s)<1/2$ is defined in \eqref{pm:parameters}.
The positivity relation \eqref{eq:entry-positive} and
\eqref{eq:positive-pressure-definition} give $\Pi\ge h$, while
\eqref{eq:ratio-quantities} gives $F=|\omega|^2/h\in L^2$; hence
\[
 |\omega|^2\Pi^{-\kappa}
 \le|\omega|^{2-2\kappa}F^\kappa.
\]
H\"older with conjugate exponents $2/\kappa$ and $2/(2-\kappa)$
therefore gives
\begin{equation}\label{eq:positive-pressure-seed}
 G_\kappa:=\int\Pi^{-\kappa}|\omega|^2\dd
 \le\|F\|_2^\kappa
       \left(\int|\omega|^{4(1-\kappa)/(2-\kappa)}\dd
       \right)^{(2-\kappa)/2}.
\end{equation}
The exponent on the right is strictly larger than $s$ precisely when
\[
 4(1-\kappa)>s(2-\kappa)
 \quad\Longleftrightarrow\quad
 \kappa<\frac{2(2-s)}{4-s}=\lambda.
\]
For a bounded weak $L^s$ function $U$ and $p>s$, the distribution
formula gives
\begin{equation}
 \int|U|^p\dd
 =p\int_0^{\|U\|_\infty}
              r^{p-1}|\{|U|>r\}|\,dr
 \le\frac{p}{p-s}\|U\|_{L^{s,\infty}}^s
                         \|U\|_\infty^{p-s}.
 \label{eq:convex-weak-strong-interpolation}
\end{equation}
Thus \eqref{eq:positive-pressure-seed} is finite: the first estimate
in \eqref{eq:ratio-global-bounds} gives $F\in L^2$, and
Lemma~\ref{lem:entry} gives $\omega\in L^\infty$, so
\eqref{eq:convex-weak-strong-interpolation} applies to the vorticity
exponent in
\eqref{eq:positive-pressure-seed}.

To transfer this seed to the other quantities, take a positive regular
value $\epsilon$ of $\Pi$, and set
\begin{align}
G_\kappa(\epsilon)&:=\int_{\{\Pi>\epsilon\}}
                      \Pi^{-\kappa}|\omega|^2\dd,\\
F_\kappa(\epsilon)&:=\int_{\{\Pi>\epsilon\}}
                      \Pi^{-\kappa}|\nabla v|^2\dd,\\
 I_\kappa(\epsilon)&:=\int_{\{\Pi>\epsilon\}}
                         \Pi^{-1-\kappa}|\nabla\Pi|^2\dd.
                    \end{align}
All three integrals are finite on this bounded superlevel set.
The first equation in \eqref{eq:entry-ns} gives
$\nabla\Pi=(v\cdot\nabla)v-\Delta v$.
We test this identity against $\nabla[\Psi_\epsilon(\Pi)]$, where
\[
 \Psi_\epsilon(t)
 :=\begin{cases}
 \displaystyle\int_\epsilon^t r^{-1-\kappa}\,dr,
       &t>\epsilon,\\
 0,&t\le\epsilon.
 \end{cases}
\]
The composition $\Psi_\epsilon(\Pi)$ is locally Lipschitz and
compactly supported, since $\Pi(x)\to0$ at infinity. Its weak
gradient is
\[
 \nabla[\Psi_\epsilon(\Pi)]
 =\mathbf1_{\{\Pi>\epsilon\}}
      \Pi^{-1-\kappa}\nabla\Pi.
\]
By approximation and integration by parts,
$\diver\Delta v=0$ gives
\[
 \int\Delta v\cdot\nabla[\Psi_\epsilon(\Pi)]\dd=0.
\]
Hence
\begin{align*}
 I_\kappa(\epsilon)
 &=\int_{\{\Pi>\epsilon\}}\Pi^{-1-\kappa}
                \nabla\Pi\cdot(v\cdot\nabla)v\dd\\
 &\le I_\kappa(\epsilon)^{1/2}
       \left(\int_{\{\Pi>\epsilon\}}
                    \Pi^{-1-\kappa}|v|^2|\nabla v|^2\dd\right)^{1/2}\\
 &\le\bigl(2I_\kappa(\epsilon)F_\kappa(\epsilon)\bigr)^{1/2}.
\end{align*}
The last inequality uses $|v|^2\le2\Pi$ from
\eqref{eq:entry-positive} and \eqref{eq:positive-pressure-definition}.
Division when $I_\kappa(\epsilon)>0$, with the zero case immediate,
gives $I_\kappa(\epsilon)\le2F_\kappa(\epsilon)$.

Next use the first identity in \eqref{eq:convex-pde}, namely
$\Delta\Pi=|\nabla v|^2-|\omega|^2$, and integrate
$\Pi^{-\kappa}\Delta\Pi$ on $\{\Pi>\epsilon\}$.
The outward normal again points toward decreasing $\Pi$, so
\begin{align*}
 F_\kappa(\epsilon)-G_\kappa(\epsilon)
 &=\int_{\{\Pi>\epsilon\}}\Pi^{-\kappa}\Delta\Pi\dd\\
 &=\kappa I_\kappa(\epsilon)
       -\epsilon^{-\kappa}\int_{\{\Pi=\epsilon\}}|\nabla\Pi|\,dS\\
 &\le2\kappa F_\kappa(\epsilon).
\end{align*}
Because $\kappa<\lambda<1/2$, this proves
$(1-2\kappa)F_\kappa(\epsilon)\le G_\kappa$ before any limit is taken.
Letting regular $\epsilon\downarrow0$ and using monotone convergence
then yields
\begin{equation}\label{eq:convex-seed}
 \begin{aligned}
 F_\kappa&:=\int\Pi^{-\kappa}|\nabla v|^2\dd
                        \le\frac{G_\kappa}{1-2\kappa},\\
 I_\kappa&:=\int\Pi^{-1-\kappa}|\nabla\Pi|^2\dd
                        \le2F_\kappa,\\
 H_\kappa&:=\int\Pi^{-1-\kappa}|\nabla h|^2\dd
                        \le2I_\kappa+4F_\kappa<\infty.
 \end{aligned}
\end{equation}
For the last bound, differentiate the definitions
\eqref{eq:entry-bernoulli-definition} and
\eqref{eq:positive-pressure-definition}, and use
$|v|^2\le2\Pi$ from \eqref{eq:entry-positive} to obtain
\[
 |\nabla h|^2
 \le2|\nabla\Pi|^2+2|\nabla(|v|^2/2)|^2
 \le2|\nabla\Pi|^2+4\Pi|\nabla v|^2.
\]

\begin{lemma}[Convex entropy trace]\label{lem:convex-entropy}
Let $f$ and $W$ be defined by \eqref{eq:positive-pressure-definition}
and \eqref{eq:ratio-quantities}, respectively.
For $0<\kappa<\lambda$, $1<\ell<2$, and
$\theta:=\kappa(2-\ell)/2$, one has
\begin{equation}\label{eq:convex-entropy-trace}
 \int f^{-\theta}|W|^\ell\,d\mu_t\le Ct^\theta.
\end{equation}
\end{lemma}

\begin{proof}
Using
 the definition of $\mu_t$ in
\eqref{eq:ratio-band-definitions}, and the weighted coarea formula
\eqref{eq:weighted-coarea}, we obtain
\begin{align}\label{10030122}
 \int f^{-\theta}|W|^\ell\,d\mu_t
 &= t \int_{t < h < 2t} f^{-\theta}|W|^\ell  h^{-2}|\nabla h|^2\dd \notag\\
 &= t\int_t^{2t}\tau^{\theta -\ell -2}
       \left(
       \int_{h=\tau}\Pi^{-\theta}
       |\omega|^\ell  |\nabla h|\,dS
       \right)d\tau.
\end{align}
It therefore suffices to prove, uniformly over positive regular
values $\tau$,
\[
 \int_{h=\tau}\Pi^{-\theta}
       |\omega|^\ell|\nabla h|\,dS\le C\tau^\ell.
\]
Indeed, substituting this bound into \eqref{10030122} gives
$Ct\int_t^{2t}\tau^{\theta-2}\,d\tau\le Ct^\theta$.
We prove the surface estimate in seven steps.

\smallskip
\noindent\textbf{Step 1. Construction of a uniformly convex profile
with uniform growth bounds.}

Start from the convex continuation
\begin{equation}\label{eq:convex-profile-R}
 \phi(z):=
 \begin{cases}
  z^\ell,&0\le z\le1,\\
  \dfrac{\ell}{2}z^2+1-\dfrac{\ell}{2},&z\ge1.
 \end{cases}
\end{equation}
The values and first derivatives of the two branches agree at $z=1$ and
\begin{align}\label{10030053}
z\phi'(z)-\phi(z)\ge(\ell-1)z^\ell.
\end{align}
To regularize the singularity at the vorticity zero set, let
$0<\varepsilon\le1$ and define
\[
 b_\varepsilon(z):=(z^2+\varepsilon^2)^{\ell/2}-\varepsilon^\ell,
 \qquad
 \phi_\varepsilon(z):=
 \begin{cases}
 b_\varepsilon(z),&0\le z\le1,\\
 c_\varepsilon z^2+d_\varepsilon,&z\ge1,
 \end{cases}
\]
where $c_\varepsilon$ and $d_\varepsilon$ are chosen so that
$\phi_\varepsilon$ is $C^1$ at $z=1$.
Then $\phi_\varepsilon$ is smooth as a radial function at $z=0$,
locally $C^{1,1}$, and satisfies the uniform convexity bound
\[
 \phi_\varepsilon''(z)\ge c_\ell,
 \qquad
 \frac{\phi_\varepsilon'(z)}{z}\ge c_\ell
 \quad(z>0),
\]
for some $c_\ell>0$ independent of $\varepsilon$.
Second derivatives are understood almost everywhere; the identities
below use the weak chain rule for locally $C^{1,1}$ functions.
Moreover, $\phi_\varepsilon\to\phi$ together with the first derivative
as $\varepsilon\downarrow0$.
The only auxiliary profile quantity needed later is
\begin{equation}
 \mathcal E_\varepsilon(z)
 :=z\phi_\varepsilon'(z)-\phi_\varepsilon(z).
 \label{eq:convex-profile-excess}
\end{equation}
By \eqref{eq:convex-profile-excess},
$\mathcal E_\varepsilon(0)=0$ and
$\mathcal E_\varepsilon'(z)=z\phi_\varepsilon''(z)\ge0$,
so $\mathcal E_\varepsilon\ge0$.
A direct computation also gives the uniform bounds
\begin{gather}\label{eq:convex-profile-bounds}
 \mathcal E_\varepsilon(z)\ge0,\qquad
 |z\phi_\varepsilon'-2\phi_\varepsilon|
       \le C\min\{z,1\},\\
 |z\phi_\varepsilon''-\phi_\varepsilon'|
       +|z(z\phi_\varepsilon''-\phi_\varepsilon')|\le C,\qquad
 0\le\phi_\varepsilon(z)\le C(z+z^2),\qquad
 |z\phi_\varepsilon'(z)|\le C(z+z^2),
\end{gather}
where $C$ is independent of $\varepsilon$.

\smallskip
\noindent\textbf{Step 2. The weighted entropy Hessian is coercive up
to a pressure-gradient error.}
Recalling \eqref{eq:ratio-concave-denominator}, let us fix a regular height $\tau>0$ and set
\begin{equation}\label{eq:convex-entropy-variables}
 u:=\psi_\tau(h)=h+\min(h,\tau),\qquad
 Y:=\frac{\omega}{u},\qquad
 z:=\Pi^{\kappa/2}|Y|,
\end{equation}
and
\begin{equation}\label{eq:convex-entropy-definition}
 \Phi(r,Y):=r^{-\kappa}
 \phi_\varepsilon(r^{\kappa/2}|Y|),
 \qquad
 Z:=u\Phi(\Pi,Y),
 \qquad r>0.
\end{equation}
The quadratic specialization $\phi_\varepsilon(z)=z^2$ of
\eqref{eq:convex-entropy-definition} gives
$Z=u|Y|^2=|\omega|^2/u$, the quotient in
\eqref{eq:ratio-general-denominator} with $\psi=\psi_\tau$.
The lower power $1<\ell<2$ near the origin in
\eqref{eq:convex-profile-R} gives the positive exponent
$\theta=\kappa(2-\ell)/2$ in the trace estimate; its quadratic
continuation retains the growth control needed below.

Then \eqref{eq:entry-positive} gives

\begin{equation}\label{eq:convex-u-bounds}
 h\le u\le2h\le2\Pi,\qquad
 |\nabla u|\le2|\nabla h|.
\end{equation}
Write  $n:=Y/|Y|$ when $Y\ne0$.
Below $\Phi$ and its derivatives are evaluated at $(\Pi,Y)$.
Differentiation gives
\begin{equation}\label{eq:convex-profile-derivatives}
\begin{aligned}
D_Y\Phi&=\Pi^{-\kappa/2}\phi_\varepsilon'(z)n,\\
 D_Y\Phi\cdot Y-\Phi&=\Pi^{-\kappa}\mathcal E_\varepsilon(z),\\
 \Phi_r&=\frac{\kappa}{2}\Pi^{-\kappa-1}
             (z\phi_\varepsilon'-2\phi_\varepsilon),\\
 D_Y\Phi_r&=\frac{\kappa}{2}\Pi^{-\kappa/2-1}
             (z\phi_\varepsilon''-\phi_\varepsilon')n,\\
 \Phi_{rr} &=\frac{\kappa}{2}\Pi^{-\kappa-2}
       \bigl[-(\kappa+1)(z\phi_\varepsilon'-2\phi_\varepsilon)
           +\frac{\kappa}{2}z
             (z\phi_\varepsilon''-\phi_\varepsilon')\bigr],\\
 D^2_{YY}\Phi&=\phi_\varepsilon''(z)n\otimes n
             +\frac{\phi_\varepsilon'(z)}z(I-n\otimes n).
\end{aligned}
\end{equation}
At $Y=0$ these formulas have their continuous radial interpretation.

The uniform convexity of $\phi_\varepsilon$ in the $Y$ variable and
the bounds in \eqref{eq:convex-profile-bounds} imply
\[
 D^2_{YY}\Phi[\xi,\xi]\ge c_\ell|\xi|^2,
 \qquad
 |D_Y\Phi_r|\le C\Pi^{-\kappa/2-1},
 \qquad
 |\Phi_{rr}|\le C\Pi^{-\kappa-2}.
\]
Hence, by Young's inequality,
\begin{equation}\label{eq:convex-hessian}
 D^2\Phi[(\eta,\xi),(\eta,\xi)]
 \ge \frac{c_\ell}{2}|\xi|^2
      -C\Pi^{-\kappa-2}\eta^2.
\end{equation}
Thus the Hessian contributes a positive $Y$-gradient term, up to a
pressure-gradient error that will be controlled below.

\smallskip
\noindent\textbf{Step 3. The mixed derivatives cancel in the entropy
identity.}
Away from $h=\tau$, the second identity in
\eqref{eq:convex-pde} gives
$L\omega=-E\omega$. Substituting
$\omega=uY$ and applying the product rule gives
\begin{equation}
 uLY+YLu+2\nabla u\cdot\nabla Y=-uEY.
 \label{eq:convex-Y-equation}
\end{equation}
The product and chain rules also give
\begin{equation}\label{eq:convex-prechain}
\begin{aligned}
 LZ&=\Phi Lu+uL\Phi+2\nabla u\cdot\nabla\Phi,\\
 L\Phi&=\Phi_r L\Pi+D_Y\Phi\cdot LY
  +\sum_{j=1}^3D^2\Phi[
       (\partial_j\Pi,\partial_jY),(\partial_j\Pi,\partial_jY)].
\end{aligned}
\end{equation}
Substitute \eqref{eq:convex-Y-equation} for $LY$ in
\eqref{eq:convex-prechain}. The term
$-2D_Y\Phi\cdot(\nabla u\cdot\nabla Y)$ cancels the corresponding
part of $2\nabla u\cdot\nabla\Phi$. The remaining terms are
\begin{align}\label{eq:convex-chain}
 LZ={}&-\Pi^{-\kappa}\mathcal E_\varepsilon(z)Lu
       -uD_Y\Phi\cdot(EY)
       +u\Phi_r L\Pi+2\Phi_r\nabla u\cdot\nabla\Pi\nonumber\\
     &+u\sum_{j=1}^3D^2\Phi[
                 (\partial_j\Pi,\partial_jY),
                 (\partial_j\Pi,\partial_jY)].
\end{align}

\smallskip
\noindent\textbf{Step 4. The height interface produces a nonnegative
measure.}
Since $u=\psi_\tau(h)=h+\min(h,\tau)$, the first identity in
\eqref{eq:ratio-basic-equations} gives, in distributions,
\[
 Lu=-\psi_\tau'(h)|\omega|^2
    -|\nabla h|\,dS\big|_{\{h=\tau\}}.
\]
Hence, since $\mathcal E_\varepsilon\ge0$, the first term in
\eqref{eq:convex-chain} contributes only nonnegative terms, including
the interface measure
\[
 \Pi^{-\kappa}\mathcal E_\varepsilon(z)
 |\nabla h|\,dS\big|_{\{h=\tau\}}.
\]

\smallskip
\noindent\textbf{Step 5. Every remaining error term in $LZ$ is uniformly
integrable.}

By \eqref{eq:convex-profile-derivatives},
\eqref{eq:convex-entropy-variables}, and
\eqref{eq:convex-profile-bounds},
\begin{align}
 u|D_Y\Phi\cdot(EY)|
 &\le u\Pi^{-\kappa}|z\phi_\varepsilon'(z)|\,|E|\notag\\
 &\le C u\Pi^{-\kappa}(z+z^2)|E|\notag\\
 &\le C\bigl(|\omega|\Pi^{-\kappa/2}+F\bigr)|E|,
 \label{eq:convex-stretching-pointwise}
\end{align}
where $F=|\omega|^2/h$ is defined in \eqref{eq:ratio-quantities}.
By Cauchy--Schwarz, \eqref{eq:ratio-strain-energy},
\eqref{eq:entry-dirichlet}, and \eqref{eq:convex-seed},
\begin{align}
 \int u|D_Y\Phi\cdot(EY)|\dd
 &\le C\|F\|_2\|E\|_2
      +C\left(\int|\omega|^2\dd\right)^{1/2}
          \left(\int\Pi^{-\kappa}|E|^2\dd\right)^{1/2}\\
 &\le C\bigl(\|F\|_2\|E\|_2+\sqrt{DF_\kappa}\bigr).
 \label{eq:convex-stretching-error}
\end{align}

Next, we will deal with the term 
\[
u\Phi_r L\Pi.
\]
Following from \eqref{eq:ratio-operator-definition} and
\eqref{eq:convex-pde}, it holds that
\[
 L\Pi=|\nabla v|^2-|\omega|^2-v\cdot\nabla\Pi.
\]
\eqref{eq:convex-profile-derivatives},
\eqref{eq:convex-profile-bounds}, and
\eqref{eq:convex-u-bounds} give
\begin{equation}
 u|\Phi_r|\,\bigl||\nabla v|^2-|\omega|^2\bigr|
 \le C\Pi^{-\kappa}\bigl(|\nabla v|^2+|\omega|^2\bigr).
 \label{eq:convex-pressure-source-pointwise}
\end{equation}
Hence, by \eqref{eq:positive-pressure-seed} and
\eqref{eq:convex-seed},
\begin{equation}
 \int u|\Phi_r|\,\bigl||\nabla v|^2-|\omega|^2\bigr|\dd
 \le C(F_\kappa+G_\kappa).
 \label{eq:convex-pressure-source-error}
\end{equation}
For the transport part, the sharper bound
$|z\phi_\varepsilon'-2\phi_\varepsilon|\le Cz$ in
\eqref{eq:convex-profile-bounds}, together with
\eqref{eq:convex-profile-derivatives} and
\eqref{eq:convex-entropy-variables}, yields
\[
 u|\Phi_r||v\cdot\nabla\Pi|
 \le C|\omega|\Pi^{-\kappa/2-1}|v||\nabla\Pi|.
\]
Using $|v|^2\le2\Pi$ by
\eqref{eq:entry-positive}, we have
\[
 u|\Phi_r||v\cdot\nabla\Pi|
 \le C|\omega|\Pi^{-(1+\kappa)/2}|\nabla\Pi|.
\]
Hence Cauchy--Schwarz gives
\begin{align}
 \int u|\Phi_r||v\cdot\nabla\Pi|\dd
 &\le C
 \left(\int|\omega|^2\dd\right)^{1/2}
 \left(\int\Pi^{-1-\kappa}|\nabla\Pi|^2\dd\right)^{1/2}\notag\\
 &=C\sqrt{DI_\kappa},
 \label{eq:convex-pressure-transport-error}
\end{align}
where we used \eqref{eq:entry-dirichlet} and
\eqref{eq:convex-seed}.

Next,
for the remaining mixed derivative term,
\eqref{eq:convex-profile-derivatives},
\eqref{eq:convex-profile-bounds}, and
\eqref{eq:convex-u-bounds} give
\[
 |\Phi_r||\nabla u||\nabla\Pi|
 \le C\Pi^{-1-\kappa}|\nabla h||\nabla\Pi|.
\]
Therefore, by Cauchy--Schwarz and \eqref{eq:convex-seed},
\begin{equation}
 \int|\Phi_r||\nabla u||\nabla\Pi|\dd
 \le C\sqrt{H_\kappa I_\kappa}.
 \label{eq:convex-mixed-derivative-error}
\end{equation}

Finally the negative term in \eqref{eq:convex-hessian}, together
with \eqref{eq:convex-u-bounds} and \eqref{eq:convex-seed}, gives
\begin{equation}
 \int u\Pi^{-\kappa-2}|\nabla\Pi|^2\dd
 \le2I_\kappa.
 \label{eq:convex-hessian-error}
\end{equation}
All these bounds are uniform in $\tau$ and $\varepsilon$.

\smallskip
\noindent\textbf{Step 6. The spatial cutoff can be removed.}
By \eqref{eq:convex-entropy-definition},
\eqref{eq:convex-entropy-variables}, and the growth bound
for $\phi_\varepsilon$ in \eqref{eq:convex-profile-bounds},
\begin{align*}
 0\le Z
 &=u\Pi^{-\kappa}\phi_\varepsilon(z)\\
 &\le Cu\Pi^{-\kappa}(z+z^2)\\
&\le C\bigl(|\omega|\Pi^{-\kappa/2}+F\bigr),
\end{align*}
where the last inequality follows from
$u\ge h$ in \eqref{eq:convex-u-bounds} and
$F=|\omega|^2/h$ in \eqref{eq:ratio-quantities}.
Consequently, using the definition of $G_\kappa$ in
\eqref{eq:positive-pressure-seed},
\[
 \|Z\|_2^2
 \le C\bigl(\|F\|_2^2+G_\kappa\bigr).
\]
Thus $Z\in L^2(\mathbb R^3)$ uniformly in $\tau$ and $\varepsilon$.
Use the cutoff $\chi$ from the proof of
Lemma~\ref{lem:ratio-budget}, and set $\chi_N(x):=\chi(x/N)$.
Then
\begin{align*}
 \left|\int Z\Delta\chi_N\dd\right|
 &\le CN^{-1/2}\|Z\|_2\longrightarrow0,\\
 \left|\int Zv\cdot\nabla\chi_N\dd\right|
 &\le C\|Z\|_{L^2(\{N<|x|<2N\})}
         \|v\|_{L^6(\{N<|x|<2N\})}\longrightarrow0.
\end{align*}
Here $v\in L^6$ is exactly the first estimate in
\eqref{eq:entry-strong-velocity}.

Substituting \eqref{eq:convex-hessian},
\eqref{eq:convex-stretching-error},
\eqref{eq:convex-pressure-source-error},
\eqref{eq:convex-pressure-transport-error},
\eqref{eq:convex-mixed-derivative-error}, and
\eqref{eq:convex-hessian-error} into
\eqref{eq:convex-chain},
and the nonnegative interface contribution from Step~4 into
\eqref{eq:convex-chain}, we multiply the resulting inequality by
$\chi_N$ and integrate over $\mathbb R^3$.  Using the cutoff estimates
above and letting $N\to\infty$, we obtain
\begin{equation}\label{eq:convex-interface}
 \int_{h=\tau}\Pi^{-\kappa}
     \mathcal E_\varepsilon\! \left(z\right)|\nabla h|\,dS
   +\frac{c_\ell}{2}\int u|\nabla Y|^2\dd
 \le C,
\end{equation}
where $C$ is independent of $\tau$ and $\varepsilon$. 

\smallskip
\noindent\textbf{Step 7. The singular profile limit yields the
weighted trace estimate.}

Let $\varepsilon\downarrow0$ in \eqref{eq:convex-interface}. Since
$\phi_\varepsilon\to\phi$ and
$\phi_\varepsilon'\to\phi'$ pointwise, Fatou's lemma and
\eqref{10030053} give
\begin{align*}
 C
 &\ge \int_{h=\tau}\Pi^{-\kappa}
       \bigl(z\phi'(z)-\phi(z)\bigr)|\nabla h|\,dS\\
 &\ge (\ell-1)\int_{h=\tau}\Pi^{-\kappa}
       z^\ell|\nabla h|\,dS.
\end{align*}
On $\{h=\tau\}$, we have
$z=\Pi^{\kappa/2}|\omega|/(2\tau)$. Hence
\[
 \int_{h=\tau}\Pi^{-\theta}
       |\omega|^\ell|\nabla h|\,dS\le C \tau^{\ell},
 \qquad
 \theta:=\kappa-\frac{\kappa\ell}{2}
        =\frac{\kappa(2-\ell)}2.
\]
 Using
\eqref{10030122}, we obtain
\begin{align*}
 \int f^{-\theta}|W|^\ell\,d\mu_t
 &=  t\int_t^{2t}\tau^{\theta -\ell -2}
       \left(
       \int_{h=\tau}\Pi^{-\theta}
       |\omega|^\ell  |\nabla h|\,dS
       \right)d\tau\\
 &\le Ct\int_t^{2t}\tau^{\theta-2}\,d\tau\\
 &=Ct^\theta\int_1^2r^{\theta-2}\,dr
 \le Ct^\theta.
\end{align*}
This proves \eqref{eq:convex-entropy-trace}.

\end{proof}

\section{Integral feedback and proof of the main theorem}
\label{sec:feedback}

The proof gives a slightly larger interval than the convenient bound
$s\le1.87$ in Theorem~\ref{thm:main}. To describe it, let $s_*$
denote the unique root in $(9/5,2)$ of
\begin{equation}\label{eq:thresholds}
 s_*^3-2s_*^2-19s_*+36=0.
\end{equation}
The cubic formula and numerical evaluation give the three real roots
\[
 -4.32246081230217\ldots,\qquad
 s_*=1.870963541397\ldots,\qquad
 4.45149727090499\ldots .
\]

\begin{proposition}[Liouville theorem in the strict range]
\label{prop:pressure-closure}
Let $(v,p)$ be a smooth solution of \eqref{eq:NS} satisfying
\eqref{eq:far-field}. If $\omega\in L^{s,\infty}(\mathbb R^3)$ and
\[
 \frac95<s<s_*,
\]
then $v\equiv0$ and $p$ is constant.
\end{proposition}

\begin{proof}
Suppose, to the contrary, that $v\not\equiv0$.
Then $D>0$, and Lemma~\ref{lem:entry-bernoulli} gives $h>0$.
Choose $0<\kappa<\lambda$ and $1<\ell<2$, and set
\[
 \theta:=\frac{\kappa(2-\ell)}2.
\]
For the first two steps, fix $d>0$ so that
\begin{equation}
 d<\frac{\theta}{(1+\delta)\ell};
 \label{eq:convex-d-upper}
\end{equation}
Step~3 verifies that these three parameters can be chosen
simultaneously.

\smallskip
\noindent\textbf{Step 1. A low-vorticity core carries the Bernoulli mass.}

Recall from \eqref{eq:ratio-band-definitions} that
$\mathcal S_t=\{t<h<2t\}$ and that $\mu_t$ is supported on
$\mathcal S_t$. Define
\[
 \Omega_t:=\mathcal S_t\cap\{|W|\le t^d\}.
\]
For $K\ge4$, the set
$\mathcal S_t\setminus\Omega_t$ is contained in the union of
$\{f>K\}$ and $\{|W|>t^d,\ f\le K\}$. On the second set,
$f^{-\theta}|W|^\ell\ge K^{-\theta}t^{d\ell}$.
Consequently, the pressure-tail estimate \eqref{eq:convex-tail} and
the trace estimate \eqref{eq:convex-entropy-trace} give
\begin{align*}
 \mu_t(\mathcal S_t\setminus\Omega_t)
 &\le \mu_t(f>K)
   +K^\theta t^{-d\ell}
      \int f^{-\theta}|W|^\ell\,d\mu_t\\
 &\le Ct^{-\delta}K^{-1-\delta}
      +Ct^{\theta-d\ell}K^\theta.
\end{align*}
Choose
\[
 K:=t^{-(\delta+\theta-d\ell)/(1+\delta+\theta)}.
\]
The upper bound \eqref{eq:convex-d-upper} implies
$d\ell<\theta$, so $K\to\infty$ and hence $K\ge4$ for all small
$t$. The two terms in the preceding estimate have the same power of
$t$, namely
\[
 \mu_t(\mathcal S_t\setminus\Omega_t)
 \le C
 t^{[\theta-(1+\delta)d\ell]/(1+\delta+\theta)}
 \longrightarrow0,
\]
where positivity of the exponent follows again from
\eqref{eq:convex-d-upper}. Since $\mu_t$ is supported on
$\mathcal S_t$ and \eqref{eq:convex-mass} gives
$\mu_t(\mathbb R^3)\to D/2$, subtracting the vanishing complement
yields
\begin{equation}
 \mu_t(\Omega_t)\longrightarrow\frac D2.
 \label{eq:convex-core-mass}
\end{equation}

\smallskip
\noindent\textbf{Step 2. The Bernoulli-gradient identity turns the core mass into a
pressure--vorticity product.}

On $\Omega_t$ one has $|\omega|=h|W|\le2t^{1+d}$.
By the layer-cake formula,
\[
\int_{\Omega_t}|\omega|^2\dd
 \le\int_{\mathbb R^3}\min\{|\omega|^2,4t^{2(1+d)}\}\dd
 =2\int_0^{2t^{1+d}}r\,|\{|\omega|>r\}|\,dr.
\]
Using \eqref{eq:entry-weak-vorticity-norm},
\begin{align}
 \int_{\Omega_t}|\omega|^2\dd
 &\le
 2\|\omega\|_{L^{s,\infty}}^s
 \int_0^{2t^{1+d}}r^{1-s}\,dr \notag\\
 &\le C t^{(2-s)(1+d)}.
 \label{eq:convex-core-energy}
\end{align}
The quotient-gradient estimate in \eqref{eq:ratio-global-bounds} and
dominated convergence imply
\begin{equation}
 \int_{\mathcal S_t}h|\nabla W|^2\dd
 \le\int_{\{h<2t\}}h|\nabla W|^2\dd\longrightarrow0.
 \label{eq:convex-all-layers}
\end{equation}
Indeed, $h>0$ makes $\mathbf1_{\{h<2t\}}\to0$ pointwise, while
$h|\nabla W|^2$ is integrable by \eqref{eq:ratio-global-bounds}.

Substituting $\omega=hW$ into the third identity in
\eqref{eq:convex-pde} and dividing by $h$ yield
\[
 \nabla\log h
 =\curl W+\nabla\log h\times W-v\times W.
\]
Taking the scalar product with $\nabla\log h$ eliminates the middle
term and gives
\begin{equation}
 |\nabla\log h|^2
 =\nabla\log h\cdot\curl W
  -\nabla\log h\cdot(v\times W).
 \label{eq:core-lamb-pointwise}
\end{equation}
Multiply \eqref{eq:core-lamb-pointwise} by $t$ and integrate over
$\Omega_t$. By Cauchy--Schwarz and the definition of $\mu_t$,
\begin{align*}
 t\int_{\Omega_t}|\nabla\log h|\,|\curl W|\dd
 &\le
 \left(t\int_{\Omega_t}|\nabla\log h|^2\dd\right)^{1/2}
 \left(t\int_{\Omega_t}|\curl W|^2\dd\right)^{1/2}\\
 &=\mu_t(\Omega_t)^{1/2}
 \left(t\int_{\Omega_t}|\curl W|^2\dd\right)^{1/2}.
\end{align*}
Since $|\curl W|^2\le2|\nabla W|^2$ and $t<h$ on
$\mathcal S_t$,
\[
 t\int_{\Omega_t}|\curl W|^2\dd
 \le2\int_{\mathcal S_t}h|\nabla W|^2\dd.
\]
Therefore,
\[
 t\int_{\Omega_t}|\nabla\log h|\,|\curl W|\dd
 \le
 \left(
 2\mu_t(\Omega_t)
 \int_{\mathcal S_t}h|\nabla W|^2\dd
 \right)^{1/2}
 =o(1),
\]
by \eqref{eq:convex-core-mass} and
\eqref{eq:convex-all-layers}.
For the remaining term, \eqref{eq:entry-positive},
\eqref{eq:positive-pressure-definition}, and
\eqref{eq:ratio-quantities} give $|v|^2\le2\Pi=2fh$ and
$W=\omega/h$. Hence
\[
 t|\nabla\log h|\,|v\times W|
 \le
 (2tf|\nabla\log h|^2)^{1/2}
 (t|\omega|^2/h)^{1/2}.
\]
Using $t/h<1$ on $\mathcal S_t$ and applying Cauchy--Schwarz once
more, we obtain
\[
 \mu_t(\Omega_t)
 \le o(1)+
 \left(
 2\int_{\Omega_t}f\,d\mu_t
  \int_{\Omega_t}|\omega|^2\dd
 \right)^{1/2}.
\]
By \eqref{eq:convex-core-mass}, the first term can be absorbed for all
small $t$. Therefore
\begin{equation}
 \mu_t(\Omega_t)^2
 \le8
 \left(\int_{\Omega_t}f\,d\mu_t\right)
 \left(\int_{\Omega_t}|\omega|^2\dd\right).
 \label{eq:core-lamb-absorption}
\end{equation}
This argument only integrates the pointwise identity
\eqref{eq:core-lamb-pointwise}; no regularity of
$\partial\Omega_t$ is required.

\smallskip
\noindent\textbf{Step 3. One exponent condition closes the feedback.}

Apply the pressure-core estimate \eqref{eq:convex-pressure-core} with
$\Omega=\Omega_t$ in \eqref{eq:core-lamb-absorption}, and then use
\eqref{eq:convex-core-energy}. This gives
\[
 \mu_t(\Omega_t)^2
 \le Ct^{-\delta/(1+\delta)}
       \mu_t(\Omega_t)^{\delta/(1+\delta)}
       t^{(2-s)(1+d)}.
\]
By \eqref{eq:convex-core-mass},
$\mu_t(\Omega_t)\ge D/4$ for all small $t$. Dividing by the positive
power $\mu_t(\Omega_t)^{\delta/(1+\delta)}$ and raising to the power
$1+\delta$, we obtain
\begin{equation}
 D^{2+\delta}
 \le Ct^{(1+\delta)(2-s)(1+d)-\delta}.
 \label{eq:convex-feedback}
\end{equation}
The exponent on the right is positive precisely when
\[
 d>\frac{\delta}{(1+\delta)(2-s)}-1.
\]
Thus it suffices to choose
\begin{equation}
 \max\left\{0,
 \frac{\delta}{(1+\delta)(2-s)}-1\right\}
 <d<\frac{\kappa(2-\ell)}{2(1+\delta)\ell}.
 \label{eq:convex-d-choice}
\end{equation}
Such strict parameters exist whenever
\begin{align}
& \frac{\delta}{(1+\delta)(2-s)}-1 < \frac{\kappa(2-\ell)}{2(1+\delta)\ell}, \notag\\
 \Leftrightarrow& (1+\delta)(2-s)
 +\frac{(2-s)\kappa(2-\ell)}{2\ell}-\delta>0.
 \end{align}
Letting $\kappa\uparrow\lambda$ and $\ell\downarrow1$, the supremum
of the left-hand side is
\begin{equation}
 (1+\delta)(2-s)+\frac{(2-s)\lambda}{2}-\delta
 =\frac{s^3-2s^2-19s+36}{2(4-s)(3-s)}.
 \label{eq:strict-parameter-margin}
\end{equation}
For every $9/5<s<s_*$, the parameter choices above make the exponent
in \eqref{eq:convex-feedback} strictly positive. Letting $t\downarrow0$
therefore gives a contradiction unless $D=0$. Hence $v\equiv0$, and
\eqref{eq:entry-ns} then implies that $p$ is constant.
\end{proof}

\begin{proof}[Proof of Theorem~\ref{thm:main}]
The case $9/5<s\le1.87$ follows directly from
Proposition~\ref{prop:pressure-closure}, since $1.87<s_*$.

It remains to consider $s=9/5$.  By
\eqref{eq:entry-first-potential} and
\eqref{eq:entry-second-potential},
\[
 \omega\in L^{9/5,\infty}(\mathbb R^3)
 \cap L^{9/4,\infty}(\mathbb R^3).
\]
Hence
\[
 \omega\in L^\sigma(\mathbb R^3)
 \subset L^{\sigma,\infty}(\mathbb R^3)
 \qquad
 \text{for every } \frac95<\sigma<\frac94.
\]
Choosing $\sigma\in(9/5,1.87)$ and applying
Proposition~\ref{prop:pressure-closure} gives $v\equiv0$.
The Navier--Stokes equation then implies that $p$ is constant.
\end{proof}

\begin{remark}[A route toward the endpoint $s=2$]
\label{rem:finite-energy-gap}
The balance in \eqref{eq:strict-parameter-margin} identifies the
quantitative bottleneck of the present proof. As $s\uparrow2$, both
the admissible pressure exponent $\lambda=2(2-s)/(4-s)$ and the
vorticity decay exponent $2-s$ vanish. Meanwhile, the pressure loss
$t^{-\delta/(1+\delta)}$ in \eqref{eq:convex-pressure-core} tends to
$t^{-1/3}$. Improving the range therefore requires a stronger trace
estimate, a smaller pressure loss, or direct control of their coupling.

The quotient budget \eqref{eq:ratio-global-bounds} and the Bernoulli
mass identity \eqref{eq:convex-mass} use only finite energy. The same
Bernoulli-gradient argument as in \eqref{eq:core-lamb-absorption}
would thus close the endpoint proof if measurable cores
$\Omega_t\subset\mathcal S_t$ could be constructed with
\[
 \mu_t(\Omega_t)\longrightarrow D/2,
 \qquad
 \left(\int_{\Omega_t}f\,d\mu_t\right)
 \left(\int_{\Omega_t}|\omega|^2\dd\right)\longrightarrow0.
\]
Under $\omega\in L^2$, dominated convergence already gives
$\int_{\mathcal S_t}|\omega|^2\dd\to0$, since $h>0$ and
$\mathcal S_t\subset\{h<2t\}$. The missing estimate must couple
this decay to the potentially growing pressure moment.

A logarithmic replacement for $\Pi^{-\kappa}$ and direct estimates
of the pressure--vorticity product are possible directions. The
former would require new uniform estimates as $\kappa\downarrow0$;
neither logarithmic integrability nor a suitable adaptive core follows
from the present seed estimate \eqref{eq:positive-pressure-seed}.
Establishing the two displayed limits under $\omega\in L^2$ would
prove the classical finite-energy Liouville conjecture, by
Remark~\ref{rem:div-curl-equivalence}. The weak endpoint
$\omega\in L^{2,\infty}$ would additionally require an endpoint
energy-entry argument.
\end{remark}

\section*{Conflict of interest}
The author declares that he has no known competing financial interests
or personal relationships that could have appeared to influence the
work reported in this paper.

\section*{Data availability}
No data were used for the research described in this article.

\section*{Acknowledgments}
\noindent\textit{Use of generative artificial intelligence.}
ChatGPT and Codex (OpenAI, 2026) were used to assist with literature
searches, intermediate calculations, the development and organization
of the proof, and English exposition. The main ideas, including the
nonlinear feedback mechanism and the quotient-type differential
identity, were introduced and developed by the author. The author
critically reviewed the AI-assisted material and revised the
mathematical arguments and exposition, and takes full responsibility
for the accuracy of all arguments, references, and the final manuscript.

\end{document}